\documentclass[leqno, a4paper]{amsart}
\usepackage{amssymb}
\usepackage{hyperref}
\usepackage{color}
\usepackage[utf8]{inputenc}
\usepackage{lmodern}
 \newtheorem{thm}{Theorem}[section]
 \newtheorem{cor}[thm]{Corollary}
 \newtheorem{lem}[thm]{Lemma}
 \newtheorem{prop}[thm]{Proposition}
 \theoremstyle{definition}
 \newtheorem{defn}[thm]{Definition}
 \theoremstyle{rem}
 \newtheorem{rem}[thm]{Remark}
 \newtheorem{ex}[thm]{Example}
 \numberwithin{equation}{section}

\begin{document}

%
%
%
%
%
%
%
%
%

\title[generalized Drazin-Riesz invertible elements]{Generalized Drazin-Riesz invertible elements in a semi-simple Banach algebra}

\author{ Othman Abad and Hassane Zguitti }
\address{Department of Mathematics, Dhar El Mahraz Faculty of Science, Sidi Mohamed Ben Abdellah University, BO 1796 Fes-Atlas, 30003 Fez Morocco.}
\email{othman.abad@usmba.ac.ma}
\email{hassane.zguitti@usmba.ac.ma}


\begin{abstract}
We extend the notion of generalized Drazin-Riesz inverse introduced for bounded linear operators in \cite{Ziv} to elements in a complex unital semi-simple Banach algebra. Several characterizations and properties of generalized Drazin-Riesz invertible elements are given. In particular, we extend those of \cite{AbZg1,Djor,Ziv}.
\end{abstract}

\maketitle

\section{Introduction}

Throughout this paper, $\mathcal{A}$ denotes a complex semi-simple Banach algebra with unit $1$. Let $Inv(\mathcal{A})$ and $\mathcal{J}(\mathcal{A})$ be the set of all invertible elements in $\mathcal{A}$ and the radical Jacobson of $\mathcal{A}$, respectively. Let $\mathcal{B}$ be a Banach subalgebra of $\mathcal{A}$ with unit $1_{\mathcal{B}}$. In the case of Banach subalgebras, we generally have $1_{\mathcal{B}} \neq 1$, especially for the next Banach subalgebras $p\mathcal{A}p$ and $(1-p)\mathcal{A}(1-p)$ which are of wide importance in this paper, $p$ and $1-p$ are the units of $p\mathcal{A}p$ and $(1-p)\mathcal{A}(1-p)$ respectively.
\smallskip

 We define respectively the spectrum, the resolvent both related to $\mathcal{B}$ and the spectral radius of $a \in \mathcal{A}$ by
\begin{enumerate}
\item $\sigma_{\mathcal{B}}(a)=\{ \lambda \in \mathbb{C} \ : \ \lambda 1_{\mathcal{B}} - a \mbox{ is not invertible in } \mathcal{B} \}$, if $\mathcal{B}= \mathcal{A}$, then we write $\sigma(a)$ instead of $\sigma_{\mathcal{A}}(a)$.
\item $\rho_{\mathcal{B}}(a)=\{ \lambda \in \mathbb{C} \ : \ \lambda 1_{\mathcal{B}} - a \mbox{ is invertible in } \mathcal{B}\}=\mathbb{C} \setminus \sigma_{\mathcal{B}}(a)$, if $\mathcal{B}=\mathcal{A}$, we simply write $\rho(a)$ rather than $\rho_{\mathcal{A}}(a)$.
\item The spectral radius $r(a)= \underset{n \rightarrow \infty}{\lim} ||a^{n}||^{\frac{1}{n}}$.
\end{enumerate}
 Let $S_{\mathcal{A}}$ be the socle of $\mathcal{A}$ and $\overline{S}_{\mathcal{A}}$ be its closure in $\mathcal{A}$. An element $a\in\mathcal{A}$ is said to be {\it Fredholm} if $\pi(a)$ is invertible in $\mathcal{A}/\overline{S}_{\mathcal{A}}$, where $\pi\,:\mathcal{A}\rightarrow\mathcal{A}/\overline{S}_{\mathcal{A}}$ is the canonical quotient homomorphism. Let $\Phi_{\mathcal{A}}$ be the set of all Fredholm elements in $\mathcal{A}$. The {\it essential spectrum} $\sigma_e(a)$ of $a$ is defined by $$\sigma_e(a)=\{\lambda\in\mathbb{C}\,:\,\lambda-a\notin \Phi_{\mathcal{A}}\}.$$
Here and elsewhere, we write $\lambda-a$ instead of $\lambda 1-a$. $a\in\mathcal{A}$ is a {\it Riesz element} if $\pi(a)$ is quasinilpotent in $\mathcal{A}/\overline{S}_{\mathcal{A}}$.
\smallskip

For $a\in \mathcal{A}$, let $L_a$ and $R_a$ be the multiplication operators on $\mathcal{A}$ defined respectively by $L_a(b)=ab$ and $R_a(b)=ba$ for all $b\in\mathcal{A}$.
We denote by $asc_{l}(a)$ (resp. $asc_{r}(a)$) the ascent of $L_{a}$ (resp. the ascent of $R_{a}$), and by $dsc_{l}(a)$ (resp. $dsc_{r}$)  the descent of $L_{a}$ (resp. the descent of $R_{a}$). Recall that for a bounded linear operator $T$ acting on a complex Banach space $X$, the {\it ascent} $asc(T)$ (resp. the {\it descent} $dsc(T)$) of $T$ is the smallest non-negative integer $n$ such that the kernel (resp. the range) of $T^n$ is equal to the kernel (resp. the range) of $T^{n+1}$. If no such integer exists we set $asc(T)=\infty$ (resp. $dsc(T)=\infty$). It is well known that if $asc(T)<\infty$ and $desc(T)<\infty$ then $asc(T)=dsc(T)$ (\cite{Taylor}).
\smallskip

From \cite[Corollary 3.8]{Pearlman}, we have
$$\{ a \in \Phi_{\mathcal A} \, : \, asc_{r}(a)=dsc_{r}(a) < \infty \}=\{ a \in \Phi_{\mathcal A} \, : \, asc_{l}(a)=dsc_{l}(a) < \infty\}.
$$
An element $a\in\mathcal{A}$ is said to be {\it Browder} if $a\in\Phi_{\mathcal{A}}$, $asc_l(a)<\infty$ and $dsc_l(a)<\infty$. The Browder spectrum of $a$ is defined by $$\sigma_b(a)=\{\lambda\in\mathbb{C}\,:\,\lambda-a\mbox{ is not Browder in }\mathcal{A}\}.$$
The set $p_{00}(a)=\sigma(a) \setminus \sigma_{b}(a)$ is the Riesz points of $\sigma(a)$.\\
In the case of a Banach subalgebra $\mathcal{B}$ of $\mathcal{A}$, we denote and define the Browder spectrum of an element $a \in \mathcal{B}$ by $$\sigma_{b,\mathcal{B}}(a)=\{\lambda \in \mathbb{C} \, : \, \lambda 1_{\mathcal{B}}-a \mbox{ is not a Browder element in } \mathcal{B} \}. $$
The Browder resolvent set of an element $a \in \mathcal{B}$ related to a Banach subalgebra $\mathcal{B}$ of $\mathcal{A}$ is $\rho_{b,\mathcal{B}}(a)=\mathbb{C}\setminus\sigma_{b,\mathcal{B}}(a).$\\
Following \cite[Theorem 3.9]{Ying},  $a\in\mathcal{A}$ is Browder such that $a \in \Phi_{\mathcal A}$ and $0 \notin acc \, \sigma(a)$. Then
$$\sigma_b(a)=\sigma_e(a)\cup acc\,\sigma(a).$$
Here and elsewhere for $\Lambda\subset\mathbb{C}$, $iso\,\Lambda$, $acc\,\Lambda$ and $int\,\Lambda$ denote respectively, the set of isolated points, the set of accumulation points and the interior of $\Lambda$. We denote by $D(\lambda,r)$ the open disc centered at $\lambda\in\mathbb{C}$ and with radius $r>0$ and its corresponding closed disc by $\overline{D}(\lambda,r)$.
\smallskip

Let $p\in\mathcal{A}$ be an idempotent. For each $a\in\mathcal{A}$,  $$a=(1-p)a(1-p)+(1-p)ap+pa(1-p)+pap.$$ Then $a$ can be represented as follows
\begin{equation}\label{Pier}
a=\left(\begin{array}{cc}
   (1-p)a(1-p) &(1-p)ap \\
pa(1-p) &pap
\end{array}\right)_p.
\end{equation}
Moreover, if $p$ commute with $a$ then
\begin{equation} \label{Pier2}
a=\left(\begin{array}{cc}
   a(1-p) &0 \\
0 &ap
\end{array}\right)_p.
\end{equation}
Also, if $ap \in p\mathcal{A}p$ (resp. $a(1-p) \in (1-p)\mathcal{A}(1-p)$) is invertible in $p\mathcal{A}p$ (resp. $(1-p)\mathcal{A}(1-p)$), we denote its inverse by  $(a)^{-1}_{p\mathcal{A}p}$ (resp. $(a)^{-1}_{(1-p)\mathcal{A}(1-p)}$).\\
It is well-known that for an idempotent $p$ in $\mathcal{A}$, and for every element $a$ commuting with $p$, we have
\begin{equation}
\rho(a)=\rho_{(1-p)\mathcal{A}(1-p)}(a(1-p)) \cap \rho_{p\mathcal{A}p}(ap).
\end{equation}
See \cite[Lemma 2.1]{Djor} for $\rho(a) \subset \rho_{p\mathcal{A}p}(ap)\cap \rho_{(1-p)\mathcal{A}(1-p)}(a(1-p))$, for the other inclusion, we consider $\lambda \in \rho_{p\mathcal{A}p}(ap) \cap \rho_{(1-p)\mathcal{A}(1-p)}(a(1-p))$ and we take
$$(\lambda-a)^{-1}=(\lambda p- ap)_{p\mathcal{A}p}^{-1}p+(\lambda(1-p)-a(1-p))_{(1-p)\mathcal{A}(1-p)}^{-1}(1-p).)$$
As a consequence, if $pa=ap$ then $\sigma(a)=\sigma_{p\mathcal{A}p}(ap) \cup \sigma_{(1-p)\mathcal{A}(1-p)}(a(1-p))$.
\smallskip

Let $a\in\mathcal{A}$ and assume that there exists a spectral set $\sigma$ of $a$. Then the spectral idempotent $p_\sigma$ of $a$ corresponding to $\sigma$ is $$p_\sigma={1\over 2\pi i}\int_\Gamma(\lambda-a)^{-1}d\lambda,$$ where $\Gamma$ is a contour surrounding $\sigma$. It is easy to see that $p_\sigma$ commutes with each element which commutes with $a$. Moreover $a$ has the form (\ref{Pier2}).
\smallskip

An element $a\in\mathcal{A}$ is {\it Drazin invertible} such that there exists $b\in\mathcal{A}$ satisfying \begin{equation}\label{equDra} ab=ba,\,bab=b\mbox{ and } a-a^2b\mbox{ is nilpotent}.\end{equation}
The element $b$ is unique if it exists and we denote it by $b=a^D$ and we call it the {\it Drazin inverse} of $a$ \cite{Dra}. If $a$ is not invertible then $a$ has a Drazin inverse if and only if zero is a pole of $a$.
\smallskip

The class of the Drazin invertible elements was extended by Koliha \cite{Kol1}. An element $a\in\mathcal{\mathcal A}$ is {\it generalized Drazin invertible} such that there exists an element $b\in\mathcal{\mathcal A}$ satisfying
\begin{equation}\label{eqgDr} ab=ba,\,bab=b\mbox{ and }a-a^2b\mbox{ is quasinilpotent}.
\end{equation}
Such element $b$  is unique if exists and will be denoted by $b=a^{gD}$ and called the {\it generalized Drazin inverse} of $a$. If $a$ is not invertible then $a$ is Drazin invertible if and only if zero is an isolated point of $\sigma(a)$, (\cite{Kol1}).\\
The element $p=1-aa^{gD}$ is the spectral idempotent of $a$ corresponding to $\{0\}$. Then $a$ and $a^{gD}$ are given by
\begin{equation}\label{PiergD}
a=\left(\begin{array}{cc}
   a(1-p) &0 \\
0 &ap
\end{array}\right)_{p}\mbox{ and }a^{gD}=\left(\begin{array}{cc}
   (a(1-p))^{-1} &0 \\
0 &0
\end{array}\right)_{p}.
\end{equation}
\smallskip
The (generalized) Drazin inverse has applications in computations and matrix theory, Markov chains, linear systems theory, differential equations, and so on, see \cite{Abd,Ben,Camp,Camp2,Nash} and the references therein.

Following Živkovi\'c-Zlatanovi\'c and Cvetkovi\'c \cite{Ziv}, a bounded linear operator $T$ on a complex Banach space $X$ is said to be {\it generalized Drazin-Riesz invertible} if there exists a bounded linear operator $S$ on $X$ satisfying
\begin{equation}\label{eqgDr-R} TS=ST,\,STS=S\mbox{ and }T-T^2S\mbox{ is Riesz}.
\end{equation}
If such $S$ exists then it is called a {\it generalized Drazin-Riesz inverse} of $T$. Many properties of such generalized inverses can be found in \cite{AbZg1,Ziv}.
\smallskip

 In this paper we extend to concept of generalized Drazin-Riesz inverse to Banach algebras. Preliminary results are presented in section 2. In section 3, we introduce and study generalized Drazin-Riesz invertible elements in a semi-simple Banach algebra by extending \cite[Theorem 2.3]{Ziv}. In section 4, we focus our study to generalized Drazin-Riesz inverses. In particular, we generalize many results of \cite{AbZg1}. Section 5 is devoted to show that the class of generalized Drazin-Riesz elements forms a regularity in the sense of M\"{u}ller, which is a generalization of \cite{AbZg3}. Finally, section 6 holds results that link two generalized Drazin-Riesz invertible elements, most of these results extend those of \cite{Djor}.

\section{Preliminary results}
In this section, we show some preliminary results which are very useful in the sequel.
\begin{thm} \label{TheoRieszstability}
Let $a,b \in {\mathcal A}$. Then the following statements hold:
\begin{enumerate}
\item[i)] If $a$ and $b$ are commuting Riesz elements, then $a+b$ is a Riesz element in ${\mathcal A}$;
\item[ii)] If $a$ is a Riesz and $b$ commutes with $a$, then the product $ab$ (or/and $ba$) is a Riesz element in ${\mathcal A}$.
\end{enumerate}
\end{thm}
\begin{proof} See \cite[Theorem R.1.2]{Barnes2}.
\end{proof}

Let $\mathcal{I}_{\mathcal{A}}:=\pi^{-1}(\mathcal{J}(\mathcal{A}/\overline{S}_{\mathcal{A}}))$ be the closed two sided ideal of inessential elements in $\mathcal{A}$ \cite{Pearlman}. From \cite{Barnes1}, for a two sided ideal $M$ satisfying $S_{\mathcal A} \subset M \subset \mathcal{I}_{\mathcal A}$, we have
 $$\Phi_{\mathcal{\mathcal A}}=\{ a \in \mathcal{\mathcal A} \ : \ a+M \ \mbox{is invertible in } \mathcal{A}/M \}.$$
Hence $a$ is a Browder element if and only if $a+M$ is invertible in $\mathcal{A}/M$ and $0 \notin acc\, \sigma(a)$.
\smallskip

The following lemma is a slight generalization of \cite[Lemma 3.12]{Ying}.

\begin{lem} \label{Lemma 3.12G}
For a  two sided ideal $M$ such that $S_{\mathcal A} \subset M \subset \mathcal{I}_{\mathcal A}$. Let $a \in \mathcal{A}$, then
$$\sigma_{b}(a)=\bigcap_{\underset{au=ua}{u \in M}}\sigma(a+u)=\sigma(a) \setminus \{ \lambda \in \mathbb{C}: \ a-\lambda \in \Phi_{\mathcal A}\mbox{ and }0 \in iso\, \sigma(a-\lambda) \}.$$
\end{lem}
\begin{proof}
Let $\lambda \notin \sigma(a)\setminus \{ \lambda \in \mathbb{C}: \ a-\lambda \in \Phi_{\mathcal A} \mbox{ and }  0 \in iso\, \sigma(a-\lambda) \}$. If $\lambda \notin \sigma(a)$ then $\lambda\notin\displaystyle{\bigcap_{\underset{au=ua}{u \in M}}}\sigma(a+u)$. So assume that $\lambda \in \sigma(a)$, $a-\lambda \in \Phi_{\mathcal A}$ and $0 \in iso\, \sigma(a-\lambda)$.
Hence $a-\lambda$ is a Browder element. In other words, $$a-\lambda \in \Phi_{\mathcal A},\,asc_{l}(a-\lambda)< \infty\mbox{ and } dsc_{l}(a-\lambda)<\infty.$$ By virtue of  \cite[Theorem 3.7]{Pearlman}, $\lambda$ is a Riesz point of $\sigma(a)$.
So by \cite[Theorem 3.5]{Pearlman}, the spectral idempotent $p$ for $a$ at $\lambda$ is in $S_{\mathcal A} \subset M$ and $\lambda - (a + p)$ is invertible in $\mathcal{A}$. Hence $a(-p)=(-p)a, \ -p \in M$ and $\lambda - a +(-p)$ is invertible. Thus $$\lambda\notin\bigcap_{\underset{au=ua}{u \in M}}\sigma(a+u).$$ Therefore, $$\bigcap_{\underset{au=ua}{u \in M}}\sigma(a+u)\subseteq \sigma(a) \setminus \{ \lambda \in \mathbb{C}: \ a-\lambda \in \Phi_{\mathcal A} \mbox{ and }  0 \in iso\, \sigma(a-\lambda) \}.$$

Conversely, assume that $\lambda \notin \displaystyle{\bigcap_{\underset{au=ua}{u \in M}}}\sigma(a+u)$.
Then there exists $u \in M$ such that $ua=au$, $a-\lambda+u$ is invertible in $\mathcal{A}$. Hence, by \cite[Theorem 3.9]{Pearlman}, we conclude that $\lambda$ is a Riesz point, thus
$$
\lambda \notin \sigma(a)\setminus \sigma_{b}(a)=\sigma(a) \setminus \{ \lambda \in \mathbb{C}: \ a-\lambda \in \Phi_{\mathcal A}\mbox{ and } \ 0 \in iso\, \sigma(a-\lambda) \}.
$$
\end{proof}

The following is a generalization of \cite[Proposition 3.13]{Ying}.

\begin{prop} \label{Proposition 3.13G}
Let $M$ be a two sided ideal such that $S_{\mathcal A} \subset M \subset \mathcal{I}_{\mathcal A}$. An element $a \in {\mathcal A}$ is  Browder if and only if $a=b+c$ where $b \in Inv({\mathcal A})$, $c \in M$ and $bc=cb$.
\end{prop}
\begin{proof}
Assume that $a$ is a Browder element. Then by Lemma \ref{Lemma 3.12G}, $0 \notin \displaystyle{\bigcap_{\underset{au=ua}{u \in M}}}\sigma(a+u)$. So there exists an element $u \in M$ such that $au=ua$ and $a+u$ is invertible in ${\mathcal A}$. Set $b=a+u$ and $c=-u$, then $a=b+c$.
\smallskip

Conversely, let $a=b+c$ where $b \in Inv({\mathcal A})$, $c \in M$ and $bc=cb$. Then $a-c$ is invertible, $-c \in M$ and $b(-c)=(-c)b$. Therefore
$$0 \notin \bigcap_{\underset{au=ua}{u \in M}}\sigma(a+u).$$
Which means that $a$ is a Browder element.
\end{proof}

If $T: {\mathcal A} \longrightarrow \mathcal{B}$ is a bounded homomorphism of normed algebras $\mathcal{A}$ and $\mathcal{B}$, then $a \in \mathcal{A}$ is  called {\it $T$-Browder} if
$$a \in \{ c+d \, : \, c \in Inv(\mathcal{A}), Td=0\mbox{ and }cd=dc \}.$$

\begin{lem} \label{lem1}
Let $a$ and $b \in \mathcal{\mathcal A}$ such that $ab=ba$. Then
$$ab \mbox{ is Browder if and only if } a\mbox{ and }b \mbox{ are Browder }.$$
\end{lem}
\begin{proof}
By Proposition \ref{Proposition 3.13G} and taking $M=\overline{S}_{\mathcal A}$, an element $a$ is a Browder in ${\mathcal A}$ if and only if $a=b+c$ where $b \in Inv(\mathcal{A})$, $c \in \overline{S}_{\mathcal A} $ and $bc=cb$. Consequently, $a$ is  Browder in ${\mathcal A}$ if and only if  $a$ is $\pi$-Browder.
\smallskip

Now, since $a$ and $b$ commute, we apply \cite[Theorem 7.7.6]{Harte} to obtain  that  $ab$ is a Browder element in ${\mathcal A}$ if and only if $a$ and $b$ are Browder elements in ${\mathcal A}$.
\end{proof}

\begin{lem} \label{Browderdecompidempotent}
Let $a \in \mathcal{A}$ and $p$ be an idempotent such that $ap=pa$. Then
$a$ is Browder in $\mathcal{A}$ if and only if $ap$ and $a(1-p)$ are Browder elements respectively in $p\mathcal{A}p$ and $(1-p)\mathcal{A}(1-p)$.
\end{lem}
\begin{proof}
Suppose that $a$ is a Browder element in $\mathcal{A}$. If $a$ is invertible, then $ap$ and $a(1-p)$ are invertible respectively in $p\mathcal{A}p$ and $(1-p)\mathcal{A}(1-p)$, then they are Browder elements respectively in $p\mathcal{A}p$ and $(1-p)\mathcal{A}(1-p)$.\\
Now, if $a$ is Browder but not invertible. Then $a \in \Phi_{\mathcal{A}}$ and $0 \in iso\, \sigma(a)$. This is equivalent to say that $\pi(a)$ is invertible in $\mathcal{A} / \overline{S}_{\mathcal{A}}$ and $0 \in iso\, \sigma(a)$. Thus
                    $$\pi(ap)=\pi(a)\pi(p) \mbox{ and } \pi(a(1-p))=\pi(a)\pi(1-p)$$
are invertible respectively  in $\pi(p)\left(\mathcal{A}/\overline{S}_{\mathcal{A}}\right)\pi(p)$ and $\pi(1-p)\left(\mathcal{A}/\overline{S}_{\mathcal{A}}\right)\pi(1-p)$, with respective inverses $(\pi(a))^{-1}\pi(p)$ and $(\pi(a))^{-1}\pi(1-p)$.\\
Also, we have $\sigma(a)=\sigma_{p\mathcal{A}p}(ap)\cup \sigma_{(1-p)\mathcal{A}(1-p)}(a(1-p))$  with $0\in iso\, \sigma(a)$, then  $0\notin acc\, \sigma_{p\mathcal{A}p}(ap)$ and $0\notin acc\, \sigma_{(1-p)\mathcal{A}(1-p)}(a(1-p))$.
Therefore we conclude that $ap$ and $a(1-p)$ are Browder elements respectively in $p\mathcal{A}p$ and $(1-p)\mathcal{A}(1-p)$.
\smallskip

Conversely, if $ap$ and $a(1-p)$ are Browder in $p\mathcal{A}p$ and $(1-p)\mathcal{A}(1-p)$ respectively, then $ap \in \Phi_{p\mathcal{A}p}$, $a(1-p) \in \Phi_{(1-p)\mathcal{A}(1-p)}$, $0 \in iso\, \sigma_{p\mathcal{A}p}(ap)$ and $0 \in iso\, \sigma_{(1-p)\mathcal{A}(1-p)}(a(1-p))$. As $ap \in \Phi_{p\mathcal{A}p}$ and $a(1-p) \in \Phi_{(1-p)\mathcal{A}(1-p)}$, we conclude that $\pi(ap)$ and $\pi(a(1-p))$ are invertible respectively in $\pi(p)\left(\mathcal{A}/\overline{S}_{\mathcal{A}}\right)\pi(p)$ and $\pi(1-p)\left(\mathcal{A}/\overline{S}_{\mathcal{A}}\right)\pi(1-p)$. Thus, $\pi(a)=\pi(a)\pi(p)+\pi(a)\pi(1-p)$ is invertible in $\mathcal{A}/ \overline{S}_{\mathcal{A}}$. Also, as $0 \in iso\, \sigma_{p\mathcal{A}p}(ap)$ and $0 \in iso\, \sigma_{(1-p)\mathcal{A}(1-p)}(a(1-p))$, we get $0 \in iso\, \sigma(a)$.
Finally, $a \in \Phi_{\mathcal{A}}$ and $0 \in iso\, \sigma(a)$, hence $a$ is a Browder element in $\mathcal{A}$.
\end{proof}

The following result is an immediate consequence of the above result.
\begin{cor} \label{CoroDecompofBrowderSpect} Let $a\in\mathcal{A}$ and let $p\in\mathcal{A}$ be an idempotent which commutes with $a$. Then $$\sigma_b(a)=\sigma_{b,(1-p) \mathcal{A} (1-p)}(a(1-p))\cup\sigma_{b, p \mathcal{A} p}(ap).$$
\end{cor}

The next Theorem is inspired from \cite[Theorem 2.1]{Tran}.

\begin{thm}\label{thm1}
Let $a \in{\mathcal A}$ be a non invertible element, where ${\mathcal A}$ is a unital Banach algebra (not necessarily semi-simple), and let $\sigma$ be a bounded set of $\mathbb{C}$ containing $0$.
Then $\sigma$ is a spectral set of $a$ if and only if there is a nonzero idempotent $p\in\mathcal{\mathcal A}$ such that
\begin{enumerate}
\item[(i)] $pa=ap$;
\item[(ii)] $\sigma(ap)= \sigma$;
\item[(iii)] $a-\mu  - \xi p$ is invertible for all $\mu \in \sigma$ and for some/all $\xi \in \mathbb{C}$ such that $| \xi | > 2r$, where $r=\underset{\lambda \in \sigma}{\sup}|\lambda |$.
\end{enumerate}
If (i)-(iii) hold, then $p$ is the spectral idempotent of $a$ corresponding to $\sigma$.
\end{thm}
\begin{proof}
Assume that $\sigma$ is a spectral set of $a$ containing 0. Let $p=p_{\sigma}$ be the spectral idempotent corresponding to $\sigma$. Then, the condition (i) holds. \\
Also, considering  \cite[Theorem 1.2]{KolgsDr}, we have $\sigma_{p\mathcal{A}p}(ap)=\sigma$, and by \cite[Lemma 6]{Aupetit} $\sigma(ap)=\sigma_{p\mathcal{A}p}(ap)\cup\{0\}=\sigma \cup \{0\}$. As $0 \in \sigma$, we get $\sigma(ap)=\sigma$. Now, it follows from the proof of \cite[Theorem 1.4]{KolgsDr} that for $\xi \in \mathbb{C}$ such that $|\xi| > 2r$, and $\mu  \in \sigma$, with $r=\sup\{|\lambda| \ : \ \lambda \in \sigma \}$, we have $a-\mu-\xi p$ is invertible in $\mathcal{A}$.
\smallskip

Conversely, let $p$ be an idempotent satisfying (i)-(iii). Set $a_{1}=a(1-p)$ and $a_{2}=ap$, then $a=a_{1}+a_{2}$.
By (iii), there exists $\xi\in\mathbb{C}$ with $| \xi | > 2r$ such that $a- \xi p$ is invertible. Since $(a-\xi p)(1-p)=a_1$, then $a_1$ is invertible in $(1-p)\mathcal{A}(1-p)$. This ensures that $0 \notin \sigma_{(1-p)\mathcal{A}(1-p)}(a_{1})$.
\smallskip

Now let $\mu \in \sigma$, by (iii) we have $a- \mu - \xi p$ is invertible in $\mathcal{A}$, and as
$$a- \mu - \xi p = (a_{1} - \mu (1-p)) + (a_{2} - (\mu + \xi ) p)=(a + \xi p - \mu )(1-p)+(a-\mu + \xi p)p$$
we conclude that $(a + \xi p - \mu )(1-p)$ and $(a-\mu + \xi p)p$ are invertible in $(1-p)\mathcal{A}(1-p)$ and $p \mathcal{A}p$ respectively. Hence, $\sigma \subset \rho_{(1-p)\mathcal{A}(1-p)}(a_{1})$. Therefore, $\sigma_{(1-p)\mathcal{A}(1-p)}(a_{1}) \cap \sigma_{p\mathcal{A}p}(a_{2})=\emptyset$, and we have $\sigma \cup \sigma_{(1-p)\mathcal{A}(1-p)}(a_{1})=\sigma(a)$. Thus we conclude that $\sigma$ is a spectral set of $a$ containing $0$, and $p$ is the corresponding spectral idempotent of $a$ related to $\sigma$.
\end{proof}

\section{Generalized Drazin-Riesz invertible elements in a semi-simple Banach algebra}
Now we introduce  and we characterize the concept of generalized Drazin-Riesz invertible elements in a semi-simple Banach algebra.
\begin{defn}
An element $a \in \mathcal{\mathcal A}$ is said to be {\it generalized Drazin-Riesz invertible} if there exists $b\in \mathcal{\mathcal A}$ satisfying
$$ab=ba, \ bab=b\mbox{ and } a-a^2b \mbox{ is a Riesz element in }\mathcal{\mathcal A}.$$
$1-ab$ is an idempotent, called the idempotent associated with the generalized Drazin-Riesz inverse $b$ of $a$.
\end{defn}
We shall say that $a \in \mathcal{\mathcal A}$ is  {\it Riesz-quasi-polar} if there exists an idempotent $q\in\mathcal{A}$ satisfying
$$aq=qa,\,q \in (\mathcal{\mathcal A}a)\cap(a\mathcal{\mathcal A})\mbox{ and } a(1-q)\mbox{ is a Riesz element}.$$
This extend the Riesz quasi-polarity for bounded linear operators introduced in \cite{Ziv}.
\smallskip

In the following we give several characterizations for generalized Drazin-Riesz invertible elements in $\mathcal{A}$.
\begin{thm}\label{theok}
Let $a \in \mathcal{\mathcal A}$. Then the following conditions are equivalent:
\begin{enumerate}
\item[(i)] There exists a commuting idempotent $p \in \mathcal{\mathcal A}$  with $a$, such that $a+p$ is invertible and $ap$ is Riesz;
\item[(ii)] There exists a commuting idempotent $p \in \mathcal{\mathcal A}$  with $a$, such that $a_{1}=a(1-p)$ is invertible in $(1-p){\mathcal A}(1-p)$ and $a_{2}=ap$ is a Riesz element in ${\mathcal A}$ with $a=a_{1}+a_{2}$;
\item[(iii)] $a$ has a generalized Drazin-Riesz inverse in $\mathcal{A}$;
\item[(iv)] $a$ is a Riesz-quasi-polar element in $\mathcal{A}$;
\item[(v)] There exists a commuting idempotent $p \in \mathcal{\mathcal A}$  with $a$, such that $a+p$ is Browder and $ap$ is a Riesz element;
\item[(vi)] There exists a commuting idempotent $p \in \mathcal{A}$  with $a$, such that
 $a=a(1-p)+ap$, with $a(1-p)$ is Browder in $(1-p)\mathcal{A}(1-p)$ and $ap$ is a Riesz element in $\mathcal{A}$;
\item[(vii)] $0 \notin  acc \, \sigma_{b}(a)$.
\end{enumerate}
\end{thm}
\begin{proof}
(i) $\Rightarrow$ (ii): Let $p \in \mathcal{A}$ be a commuting idempotent with $a$ such that $a+p$ is invertible and $ap$ is a Riesz element. We set
 $$a_{1}=(a+p)(1-p)=a(1-p)\mbox{ and }a_{2}=ap.$$ Then $a=a_{1}+a_{2}$, $a_{1}$ is invertible in $(1-p)\mathcal{\mathcal A}(1-p)$ and $a_{2}$ is Riesz.
 \smallskip

\noindent (ii) $\Rightarrow$ (iii): Let $p \in \mathcal{A}$ be an idempotent  satisfying $ap=pa$, $a_{1}=a(1-p)$ is invertible in $(1-p)\mathcal{A}(1-p)$ and $a_{2}=ap$ is Riesz in $\mathcal{A}$.\\
Set $b=a_{1}^{-1}$ in $(1-p)\mathcal{\mathcal A}(1-p)$, (i.e $b=(a_{1})_{(1-p)\mathcal{A}(1-p)}^{-1}$).  Then
 $$ p=1-ab, ab=ba \mbox{ and } bab=b.$$
 Furthermore,
 \begin{align*}
    a-a^2b=a-a^2((a_{1})_{(1-p)\mathcal{A}(1-p)}^{-1}(1-p))&=a-a^{2}(1-p)(a_{1})_{(1-p)\mathcal{A}(1-p)}^{-1}(1-p)\\
                                               &=a-a_{1}(1-p)=a_{2}.
    \end{align*}
Hence, $a-a^2b$ is a Riesz element in ${\mathcal A}$. Therefore $b$ is generalized Drazin-Riesz inverse of $a$.
\smallskip

\noindent (iii) $\Rightarrow$ (iv): There exists $b\in \mathcal{A}$ satisfying
$$ab=ba, \ bab=b\mbox{ and } a-a^2b \mbox{ is a Riesz element in }\mathcal{A}.$$
Set $q=ab$. Then $q$ is an idempotent which satisfies $qa=aq$, $q=ba=ab \in (\mathcal{A}a) \cap (a\mathcal{A})$ and $a(1-q)=a-a^{2}b$ is Riesz.
\smallskip

\noindent (iv) $\Rightarrow$ (v): There exists  $q=q^2$ such that $qa=aq$, $q \in (a\mathcal{A})\cap(\mathcal{A}a)$ and $a(1-q)$ is Riesz. Set $p=1-q$. Hence there exist $u$ and $v \in \mathcal{A}$ such that $1-p=av=ua$. Thus
 $$(a+p)(uav+p)=1+ap=(uav+p)(a+p).$$
 Since $ap$ is Riesz, by  \cite[Corollary 4.13]{Pearlman}, $1+ap$ is Browder and according to Lemma \ref{lem1}, we conclude that $a+p$ is also Browder.
 \smallskip

\noindent (v) $\Rightarrow$ (vi): Let $p^2=p \in {\mathcal A}$ such that $ap=pa$, $a+p$ is Browder
  and $ap$ is a Riesz element. Since $a+p=(a+p)(1-p)+(a+p)p$, then $(a+p)(1-p)$ is Browder in $(1-p)\mathcal{A}(1-p)$. Also from equality $(a+p)(1-p)=a(1-p)$, we have $a(1-p)$ is Browder in $(1-p){\mathcal A}(1-p)$. By assumption, $ap$ is Riesz.
  \smallskip

\noindent  (vi) $\Rightarrow$ (vii): Assume that there exists an idempotent $q$ such that $a=a(1-q)+aq,$ where $a_{1}=a(1-q)$ is Browder in $(1-q)\mathcal{A}(1-q)$ and $a_{2}=aq$ is Riesz.\\
   The case $0 \notin acc( \sigma(a))$ is obvious. So assume that $0 \in acc\, \sigma(a)$. By  virtue of \cite[Corollary 4.8]{Pearlman}, there exists a sequence $(\lambda_{n})$ in $\sigma(a)$ representing $\sigma(a_{2})$ (i.e., $\sigma(a_{2})=\{0,\lambda_{1},\lambda_{2},...\}$) such that $\lambda_{n}$ is a Riesz point for all $n \in \mathbb{N}$ and $\lambda_{n} \longrightarrow 0,\ \mbox{as} \ n \longrightarrow \infty$, and using  \cite[Lemma 6]{Aupetit}, we get $\sigma(a_{2})=\sigma_{p\mathcal{A}p}(a_{2})\cup \{0\}$, therefore $\sigma_{p\mathcal{A}p}(a_{2})\setminus \{0\}=\{\lambda_{1},\lambda_{2},....\}$, by compacity of $\sigma_{p\mathcal{A}p}(a_{2})$, we deduce that $\sigma_{p\mathcal{A}p}(a_{2})=\sigma(a_{2})$. Hence $\sigma_{b,p\mathcal{A}p}(a_{2})=\{0\}$. \\
We have by Corollary \ref{CoroDecompofBrowderSpect}
\begin{align*}
\sigma_{b}(a) &= \sigma_{b,(1-p)\mathcal{A}(1-p)}(a(1-p))\cup \sigma_{b,p\mathcal{A}p}(ap) \\
           &= \sigma_{b,(1-p)\mathcal{A}(1-p)}(a_{1}) \cup \sigma_{b,p\mathcal{A}p}(a_{2})
\end{align*}
As $\sigma_{b,p\mathcal{A}p}(a_{2})=\{0\}$ and $0 \in \rho_{b,(1-p)\mathcal{A}(1-p)}(a_{1})$,  there exists $\epsilon >  0$ such that $$D(0,\epsilon) \setminus \{0\} \subset \rho_{b,(1-p)\mathcal{A}(1-p)}(a_{1})\cap \rho_{b,p\mathcal{A}p}(a_{2}).$$
Also by Lemma \ref{Browderdecompidempotent} we have  $\rho_{b,(1-p)\mathcal{A}(1-p)}(a_{1})\cap \rho_{b,p\mathcal{A}p}(a_{2})=\rho_{b}(a)$.
Thus $D(0,\epsilon)\setminus \{0\} \subset \rho_{b}(a)$, finally, $0 \notin acc\, \sigma_{b}(a)$.
\smallskip

\noindent (vii) $\Rightarrow$ (i): Suppose that $0 \notin  acc \, \sigma_{b}(a)$. If $0 \notin  acc \, \sigma(a)$, $a$ is generalized Drazin invertible and in particular, it is generalized Drazin-Riesz invertible element. The result follows at once from \cite[Theorem 3.1 ]{Kol1}.
 \smallskip

So assume that $0 \in acc \, \sigma(a)$. Then $0 \in acc(\sigma(a) \setminus \sigma_{b}(a))=acc\, p_{00}(a)$. If there exist $0\neq \alpha\in acc\, p_{00}(a) $, then
 $$\eta = \underset{\gamma \in (acc p_{00}(a)) \setminus \{0\} }{\inf } ( | \gamma |) >0.$$
 Indeed, if $\eta = 0$, there exists a sequence $(\mu_{n})$ of $(acc p_{00}(a)) \setminus \{0\}$ which converges to zero. But $(\mu_{n}) \subset (acc p_{00}(a)) \setminus \{0\} \subset acc \, \sigma(a) \subset \sigma_{b}(a)$, which implies $0 \in  acc \, \sigma_{b}(a)$, which is a contradiction.\\
 Thus, by taking $ \eta_{1}=\frac{\eta}{m}$ for a sufficiently large $m$, we can from a sequence of the elements of $\overline{D}(0,\eta_{1}) \cap p_{00}(a)$ which converges to $0$. Consequently, $\omega(a)=\overline{D}(0,\eta_{1}) \cap \overline{p_{00}(a)} $ and $(\sigma_{b}(a) \setminus \{0\}) \cup ( \overline{p_{00}(a)}\setminus \omega(a))$ are closed disjoint sets in the spectrum $\sigma(a)$.
 As $$\sigma(a)=[(\sigma_{b}(a) \setminus \{0\}) \cup ( \overline{p_{00}(a)}\setminus \omega(a))] \cup \omega(a),$$ we conclude that $\omega(a)$ is a spectral set of $\sigma(a)$ containing $0$. Therefore, applying Theorem \ref{thm1} for $\xi=-1$ and $\mu=0$, we obtain that $a+p_{\omega(a)}$ is invertible and $ap_{\omega(a)}$ is Riesz.
 \smallskip

 Now if $ acc \, p_{00}(a)=\{0\}$, there exists $\mu \in p_{00}(a)$ such that $|\mu|= \underset{\lambda \in p_{00}(a)}{\sup} (| \lambda |)$. For $\eta_{2}= \min(|\mu|, \frac{1}{2})$, we can choose $\omega(a)= \overline{D}(0,\frac{\eta_{2}}{2}) \cap \overline{p_{00}(a)}$. Therefore, $\omega(a)$ is a spectral set containing $0$.
Finally, by Theorem \ref{thm1}, we conclude that $a+p_{\omega(a)}$ is invertible and $ap_{\omega(a)}$ is Riesz.
\end{proof}

\begin{rem}\label{RemarkDrazin}\rm If $a\in\mathcal{A}$ is generalized Drazin-Riesz invertible with a generalized Drazin-Riesz inverse $b\in\mathcal{A}$, then $b$ is Drazin invertible and $b^D=a^2b$. Indeed, we have $bab=b$, $ab=ba$ and $a-a^2b$ is Riesz. Set $c=a^2b$. Then $bc=cb$ and $cbc=a^2b b a^2b=aba^2b=a^2b=c$. Also $b-b^2c=b-b^2a^2b=b-b=0$ is nilpotent. Therefore $b$ is Drazin invertible and $b^D=c$.
\end{rem}

 Here we give examples of generalized Drazin-Riesz elements in an semi-simple Banach algebra $\mathcal{A}$ which is different from the Banach algebra of bounded linear operators.

\begin{ex}\rm
Let $E=C([0,1])$ be the Banach algebra of all continuous complex valued functions on $[0,1]$ equipped with the supremum norm. Let $\mathcal{\mathcal A}=(E,E)$ be the J\"orgens Algebra (see \cite{Barnes} for definition and more details). Then every integral operator with a continuous kernel acting in the dual system $$<f,g>=\int_{0}^{1}f(t)g(t)dt,\, \forall f,g \in E,$$ is a compact element in $\mathcal{\mathcal A}$ (see  \cite[Exercise 5.24 (d)]{Jorgens}). Therefore, they are generalized Drazin-Riesz elements in $\mathcal{\mathcal A}$. Notice that every J\"orgens algebra is a semi-simple Banach algebra \cite{Barnes}.
\end{ex}

One might expect that we can omit the generalized Kato-Riesz decomposition (see \cite{Ziv} for the definition) in assertions (ii) and (ix) of \cite[Theorem 2.3]{Ziv}, but this is not true as shown in the following example.

\begin{ex}\rm
Let $S$ be the bilateral shift operator on $\ell^{2}(\mathbb{Z})$. It is well known that $\sigma(S)=\{ \lambda \in \mathbb{C} \ : \ |\lambda|=1 \ \}$. Hence for the operator $T=I-S$, $0 \notin int \, \sigma(T)$ (respectively  $0 \notin int \, \sigma_{b}(T)$). As $\sigma(S)=\sigma_{b}(S)$, we obtain that $0 \in acc \, \sigma_{b}(T)$, and $\sigma(T)=\sigma_{b}(T)$. Finally, according to  Theorem \ref{theok} (or \cite[Theorem 2.9]{AbZg1}), $T$ is not generalized Drazin-Riesz invertible .
\end{ex}

\section{Characterization of generalized Drazin-Riesz inverses}

The following definition is equivalent to \cite[Definition 2.2]{KolgsDr} given in more general setting, which is inspired from \cite[Definition 2.2]{Tran}. It will enable us to construct generalized Drazin-Riesz inverses for a generalized Drazin-Riesz element $a\in\mathcal{A}$.

\begin{defn} \label{def}
Let $a \in \mathcal{A}$ be a non invertible element with a spectral set $\sigma$ containing 0 and the corresponding spectral idempotent $p_{\sigma}$.  We define the \textit{Drazin inverse of} $a$ \textit{related to} $\sigma$ by
$$a^{D,\sigma}=(a-\xi p_{\sigma})^{-1}(1-p_{\sigma}),$$
where $\xi \in \mathbb{C}$ such that $| \xi | > 2r$ where $r=\underset{\lambda \in \sigma}{\sup} | \lambda |$.
\end{defn}

Under the condition $| \xi | > 2r$,  $a-\xi p_{\sigma}$ is invertible by virtue of Theorem \ref{thm1}, $a^{D,\sigma}$ is independent of the choice of the value of $\xi$ and $p_{\sigma}=1-aa^{D,\sigma}$.
\smallskip

Let $a \in \mathcal{A}$ be generalized Drazin-Riesz invertible such that $0 \in  acc \, \sigma(a)$.
From the proof of (vii) $\Rightarrow$ (i) in Theorem \ref{theok} , there exists $\eta > 0$ such that $$\sigma(a)=[(\sigma_{b}(a) \setminus \{0\}) \cup (\overline{p_{00}(a)} \setminus \omega(a)) ] \cup \omega(a),$$ where $\omega(a)=\overline{D}(0,\eta) \cap \overline{p_{00}(a)}$.
 Also,  $\omega(a)$ is a closed infinite and countable set, we represent it as follow, $\omega(a)=\{0,\lambda_{1}, \lambda_{2},...\}$, where $(|\lambda_{n}|)_{n}$ is a strictly decreasing sequence of Riesz points of $a$ that converge to 0.\\
For $n \in \mathbb{N}$, let $\sigma_{n}$ and $\sigma_{n}'$ be the closed sets defined by
$\sigma_{n}=\{ 0, \lambda_{n+1},\lambda_{n+2},...\}$ and
$$\sigma_{n}'= [(\sigma_{b}(a) \setminus \{0\}) \cup (\overline{p_{00}(a)} \setminus \omega(a)) ] \cup \{ \lambda_{1},\lambda_{2},...,\lambda_{n}\}=\sigma(a) \setminus \sigma_{n}.$$
\begin{thm} \label{expinv}
Let $a \in \mathcal{A}$ be generalized Drazin-Riesz invertible such that $0 \in  acc \, \sigma(a)$. Then for $n$ large enough, we have
$$a^{D,\sigma_{n}}=(a-p_{\sigma_{n}})^{-1}(1-p_{\sigma_{n}})$$
is a generalized Drazin-Riesz inverse for $a$.
\end{thm}
\begin{proof} For $n$ be large enough such that
    $r_{n}=\underset{\lambda \in \sigma_{n}}{\sup}|\lambda| < \frac{1}{2}$,
  $(a-p_{\sigma_{n}})^{-1}$ exists by Theorem \ref{thm1}. Moreover, we have $p_{\sigma_{n}}=1-aa^{D,\sigma_{n}}$. Now we prove that $a^{D, \sigma_{n}}$ is a generalized Drazin-Riesz inverse for $a$.
\smallskip

 It follows from Theorem \ref{thm1} that $\sigma(ap_{\sigma_{n}})=\sigma_{n}$. Since $\sigma_{n} \setminus \{0\}=\{\lambda_{i} \, : \, i \in \mathbb{N} \setminus \{1,..., n \} \}$ is a sequence of Riesz points of $a$ which converges to zero, then we conclude by \cite[Corollary 4.13]{Pearlman} that $ap_{\sigma_{n}}$ is a Riesz element. Hence $a(1-aa^{D,\sigma_{n}})=ap_{\sigma_{n}}$ is a Riesz element.
\smallskip

Also, we have $aa^{D,\sigma_{n}}=a^{D,\sigma_{n}}a$, and
  \begin{align*}
  a^{D,\sigma_{n}}aa^{D,\sigma_{n}} &=(a-p_{\sigma_{n}})^{-1}(1-p_{\sigma_{n}})a(a-p_{\sigma_{n}})^{-1}(1-p_{\sigma_{n}})\\
                                    &=(a-p_{\sigma_{n}})^{-2}(1-p_{\sigma_{n}})(a-p_{\sigma_{n}}+p_{\sigma_{n}}) \\
                                    &=(a-p_{\sigma_{n}})^{-1}(1-p_{\sigma_{n}}) \\
                                    &= a^{D,\sigma_{n}}.
   \end{align*}
   Finally, $a^{D,\sigma_{n}}$ is  a generalized Drazin-Riesz inverse of $a$.
\end{proof}
In the case where $0 \in \ iso\, \sigma_{b}(a)$, we mean by an appropriate $n_{0}$, every sufficiently large positive integer $n_{0}$ such that $\sigma_{n_{0}}$ is contained on a certain disk $D(0,r)$ ; where $r < \frac{1}{4}$, and $D(0,r) \cap \sigma_{n_{0}}'= \emptyset$. In the sequel we denote by $\Lambda_{n_{0}}$ the set $\Lambda_{n_{0}}=D(0,r)\setminus \sigma_{n_{0}}$.

\begin{thm} \label{holomorphicfctandDR}
	Let $a \in \mathcal{A}$ be generalized Drazin-Riesz invertible such that $0 \in  acc \, \sigma(a)$. Then, there exist some $n_{0} \in \mathbb{N}$ and a holomorphic function $h$ on some open neighborhood of $\sigma(a)$ such that $a^{D,\sigma_{n_{0}}}=h(a)$. Moreover, $$\sigma(a^{D,\sigma_{n_{0}}})=\{0\} \cup \{ \lambda^{-1} \ : \ \lambda \in (\sigma(a) \setminus \sigma_{n_{0}})=\sigma_{n_{0}}' \}.$$
\end{thm}
\begin{proof}
	Since $a$ is generalized Drazin-Riesz invertible, then by Theorem \ref{theok}, $0 \notin  acc \, \sigma_{b}(a)$. Hence, there is some $\varepsilon > 0$ such that $D(0,\varepsilon) \cap \sigma_{b}(a)= \emptyset$. If we choose $n_{0}$ and $\varepsilon$ such that $\sigma_{n_{0}} \cap D(0,\varepsilon)=\sigma_{n_{0}}$, and $D(0,\varepsilon) \cap \sigma_{n_{0}}'= \emptyset$, we have $D(0,\varepsilon)$ is an open neighborhood of $\sigma_{n_{0}}$, and $(\mathbb{C} \setminus \overline{D(0,\varepsilon)})$ is an open neighborhood of $\sigma_{n_{0}}'$. Now by \cite[Proposition 2.6]{KolgsDr} and the presentation before it, we obtain $a^{D,\sigma_{n_{0}}}=h(a)$, where $h$ is the holomorphic function $h$ on $\mathbb{C}$ defined by 0 in the open neighborhood $D(0,\varepsilon)$ of $\sigma_{n}$ and  $\lambda^{-1}$ in the open neighborhood $(\mathbb{C} \setminus \overline{D(0,\varepsilon)})$ of $\sigma(a)\setminus \sigma_{n_{0}}$.
	\smallskip
	
	By the spectral mapping theorem, we obtain
	$$\sigma(a^{D,\sigma_{n_{0}}})=\sigma(h(a))=h(\sigma(a))=\{0\} \cup \{ \lambda^{-1} \ : \ \lambda \in (\sigma(a) \setminus \sigma_{n_{0}})=\sigma_{n_{0}}' \}.$$
\end{proof}

\begin{cor}
	Under the conditions  of the previous theorem we have $$a^{D,\sigma_{n_0}}={1\over 2\pi i}\int_\Gamma \lambda^{-1}(\lambda-a)^{-1}d\lambda,$$ where  $\Gamma$ is a Countour of a bounded Cauchy domain $D$ containing $\sigma'_{n_0}$ such that $D \cap D(0,\epsilon)= \emptyset$,  with $\epsilon > 0$ and $\sigma_{n_0} \subset D(0,\epsilon)$.
\end{cor}

\begin{thm}\label{ordinaryresolvent} 
Suppose that $0 \in iso \ \sigma_{b}(a)$ and let $a^{D, \sigma_{n_{0}}}$ be a generalized Drazin-Riesz inverse of $a$ where $\sigma_{n_{0}}$ is a spectral set for $a$ for some appropriate $n_{0}$.
Then,
\begin{equation}\label{equ*}(\lambda - a )^{-1}=\sum_{k=1}^{+\infty} \lambda^{-k}a^{k-1}(1-aa^{D,\sigma_{n_{0}}})-\sum_{k=0}^{+\infty} \lambda^{k}(a^{D,\sigma_{n_{0}}})^{k+1},\,\,\,\forall \lambda\in \mathbb{C}\mbox{ such that }|\lambda_{n_0+1}|<|\lambda|<|\lambda_{n_0}|.\end{equation}
\end{thm}
\begin{proof}
Let $p_{\sigma_{n_{0}}}$ be the spectral idempotent of $a$ corresponding to $\sigma_{n_{0}}$. Then by taking $\xi = -1$ in Definition \ref{def} we get
$$a^{D,\sigma_{n_{0}}}=(a+p_{\sigma_{n_{0}}})^{-1}(1-p_{n_{0}}).$$
Let $\lambda \in \Lambda_{n_{0}}$. We have
\begin{align*}
\lambda - (a+p_{\sigma_{n_{0}}})&=((\lambda-1)-a)p_{\sigma_{n_{0}}}+(\lambda-a)(1-p_{\sigma_{n_{0}}}) \\
                                &= ((\lambda-1)-ap_{\sigma_{n_{0}}})p_{\sigma_{n_{0}}}+(\lambda-a)(1-p_{\sigma_{n_{0}}}).
\end{align*}
Since $\lambda \in \rho (a)$, then $(\lambda-a)$ is invertible in $\mathcal{A}$. On the other hand, if $\lambda=\lambda_{1}+i\lambda_{2}$, we have $\lambda -1 = (\lambda_{1}-1)+i\lambda_{2}$. Since $|\lambda| < \frac{1}{4}$, then $|\lambda_{1}|<\frac{1}{4}$ and so $-\frac{5}{4}<\lambda_{1}-1<-\frac{3}{4}$. Thus $|\lambda-1|>\frac{1}{4}>r$. We deduce that $(\lambda-1) \in \rho(ap_{\sigma_{n_{0}}})$, hence $(\lambda-1)-ap_{\sigma_{n_{0}}}$ is invertible in $\mathcal{A}$. Consequently, by applying \cite[Lemma 2.1]{Djor}, we have $\lambda-(a+p_{\sigma_{n_{0}}})$ is invertible in $\mathcal{A}$.\\
Also, as $\Lambda_{n_{0}}\cap \sigma_{n_{0}}= \emptyset$ and $\sigma_{n_{0}}=\sigma(ap_{\sigma_{n_{0}}})$ ( by virtue of Theorem \ref{thm1} ), we get $(\lambda - ap_{\sigma_{n_{0}}})$ is invertible for all $\lambda \in \Lambda_{n_{0}}$.\\
 Finally, applying \cite[Lemma 2.1]{Djor}, and taking into account that $p_{\sigma_{n_{0}}}=1-aa^{D,\sigma_{n_{0}}}$, we obtain for all $\lambda \in \Lambda_{n_{0}}=D(0,r) \setminus \sigma_{n_{0}}$
 \begin{align*}
 	(\lambda - a )^{-1} &= (\lambda-ap_{\sigma_{n_{0}}})^{-1}p_{\sigma_{n_{0}}} + (\lambda-(a+p_{\sigma_{n_{0}}}))^{-1}(1-p_{\sigma_{n_{0}}}) \\
 	&=(\lambda - ap_{\sigma_{n_{0}}})^{-1}p_{\sigma_{n_{0}}}-(1-\lambda (a+p_{\sigma_{n_{0}}})^{-1})^{-1}(a+p_{\sigma_{n_{0}}})^{-1}(1-p_{\sigma_{n_{0}}})\\
 	&= (\lambda - ap_{\sigma_{n_{0}}})^{-1}p_{\sigma_{n_{0}}}-(1-\lambda (a+p_{\sigma_{n_{0}}})^{-1}(1-p_{\sigma_{n_{0}}}))^{-1}(a+p_{\sigma_{n_{0}}})^{-1}(1-p_{\sigma_{n_{0}}}) \\
 	&= (\lambda - ap_{\sigma_{n_{0}}})^{-1}p_{\sigma_{n_{0}}}-(1-\lambda a^{D,\sigma_{n_{0}}})^{-1}a^{D,\sigma_{n_{0}}}.
 \end{align*}
 By Theorem \ref{holomorphicfctandDR}, we have $r(a^{D,\sigma_{n_{0}}})=|\lambda_{n_{0}}|^{-1}$. Thus, for all $\lambda \in \Lambda_{n_{0}}$ such that $|\lambda|<|\lambda_{n_{0}}|$, 
 \begin{equation}                 
 	\label{equ**} (\lambda - a )^{-1} = (\lambda - ap_{\sigma_{n_{0}}})^{-1}p_{\sigma_{n_{0}}}-\sum_{k=0}^{+\infty} \lambda^{k}(a^{D,\sigma_{n_{0}}})^{k+1}.
 \end{equation}
 On the other hand, as $r(ap_{\sigma_{n_{0}}})=|\lambda_{n_{0}+1}|$, then for all $\lambda \in \Lambda_{n_{0}}$ such that $|\lambda_{n_{0}}|<|\lambda|$ we obtain
 \begin{equation}
 	\label{equ***}(\lambda-a)^{-1}=\sum_{k=1}^{\infty} \lambda^{-k}a^{k-1}p_{\sigma_{n_{0}}}-(1-\lambda a^{D,\sigma_{n_{0}}})^{-1}a^{D,\sigma_{n_{0}}}.
 \end{equation}
 As $|\lambda_{n_{0}+1}|<|\lambda_{n_{0}}|$, and combining (\ref{equ**}) and (\ref{equ***}), we get for all $|\lambda_{n_{0}+1}|<|\lambda|<|\lambda_{n_{0}}|$ that
 $$(\lambda - a )^{-1}=\sum_{k=1}^{+\infty} \lambda^{-k}a^{k-1}p_{\sigma_{n_{0}}}-\sum_{k=0}^{+\infty} \lambda^{k}(a^{D,\sigma_{n_{0}}})^{k+1}.$$
\end{proof}

\begin{thm} \label{limitexpofDRinv}
Assume that $0 \in iso \, \sigma_{b}(a)$ and let $a^{D,\sigma_{n_{0}}}$ be a generalized Drazin-Riesz inverse of some appropriate $n_{0}$. Then
$$a^{D,\sigma_{n_{0}}}= \underset{\underset{\lambda \in \Lambda_{n_{0}}}{\lambda \longrightarrow 0}}{\lim}(a-\lambda)^{-1}(1-p_{\sigma_{n_{0}}}).$$
\end{thm}
\begin{proof}
Multiplying by $(1-p_{\sigma_{n_{0}}})$ both sides of Equality (\ref{equ**}) of the proof of Theorem \ref{ordinaryresolvent},  we obtain for all $\lambda \in \Lambda_{n_{0}}$ such that $|\lambda|$ is small enough
$$(\lambda - a )^{-1}(1-p_{\sigma_{n_{0}}})=- \sum_{k=0}^{+\infty}\lambda^{k}(a^{D,\sigma_{n_{0}}})^{k+1}(1-p_{\sigma_{n_{0}}}).$$
Finally, we deduce that
$$a^{D,\sigma_{n_{0}}}= \underset{\underset{\lambda \in \Lambda_{n_{0}}}{\lambda \longrightarrow 0}}{\lim}(a-\lambda)^{-1}(1-p_{\sigma_{n_{0}}}).$$
\end{proof}

\begin{rem}\rm
Another limit expression of $a^{D,\sigma_{n_{0}}}$ can be deduced from \cite[Theorem 3.1]{KolgsDr} which is
$$a^{D,\sigma_{n_{0}}}=\underset{\lambda \rightarrow 0}{\lim}(a(1-p_{\sigma_{n_{0}}})-\lambda)^{-1}(1-p_{\sigma_{n_{0}}}).$$
\end{rem}

The following theorem shows that the generalized Drazin-Riesz inverse of an element $a$ may be not unique.
\begin{thm} \label{non-uniquenessDRinv}
Let $a \in \mathcal{A}$ be generalized Drazin-Riesz invertible with $0 \in  acc \, \sigma(a)$. Then for $n_{0}$ and $n_{1} \in \mathbb{N}$ such that $n_{0}<n_{1}$ and $r_{n_{0}}< \frac{1}{2}$, we have
$$a^{D,\sigma_{n_{0}}} \neq a^{D, \sigma_{n_{1}}}.$$
\end{thm}
\begin{proof}
Let $n_{0},n_{1} \in \mathbb{N}$ such that $n_{0}<n_{1}$ and $r_{n_{0}}< \frac{1}{2}$, then $\sigma_{n_{1}}=\{0,\lambda_{n_{1}+1},....\}$ and $\sigma_{n_{0}}=\{0, \lambda_{n_{0}+1},....\}$. Hence $p_{\sigma_{n_{0}}}=e_{n_{0}+1}+e_{n_{0}+2}+...+e_{n_{1}}+p_{\sigma_{n_{1}}}$, with $e_{i}$ is the corresponding idempotent of $a$ in $\{\lambda_i\}$, for $i=n_0+1,\cdots,n_1$. Thus $\sum_{i=n_{0}+1}^{n_{1}}e_{i}=p_{\sigma_{n_{0}} \setminus \sigma_{n_{1}}}$, and $\sigma(ap_{\sigma_{n_{0}} \setminus \sigma_{n_{1}}})=\sigma_{n_{0}} \setminus \sigma_{n_{1}}$. Consequently, $ap_{\sigma_{n_{0}} \setminus \sigma_{n_{1}}}$ is invertible.\\
Now using Theorem \ref{limitexpofDRinv}, we get
$$a^{D,\sigma_{n_{1}}}=\underset{\underset{\lambda \in \Lambda_{n_{1}}}{\lambda \longrightarrow 0}}{\lim}(a-\lambda)^{-1}(1-p_{\sigma_{n_{1}}})=\underset{\underset{\lambda \in \Lambda_{n_{1}} \cap \Lambda_{n_{0}}}{\lambda \longrightarrow 0}}{\lim}(a-\lambda)^{-1}(1-p_{\sigma_{n_{1}}}),$$
and
$$a^{D,\sigma_{n_{0}}}=\underset{\underset{\lambda \in \Lambda_{n_{0}}}{\lambda \longrightarrow 0}}{\lim}(a-\lambda)^{-1}(1-p_{\sigma_{n_{0}}})=\underset{\underset{\lambda \in \Lambda_{n_{1}} \cap \Lambda_{n_{0}}}{\lambda \longrightarrow 0}}{\lim}(a-\lambda)^{-1}(1-p_{\sigma_{n_{1}}}-p_{\sigma_{n_{0}}\setminus \sigma_{n_{1}}}).$$
Therefore, $$a^{D,\sigma_{n_{0}}}=a^{D,\sigma_{n_{1}}}-\underset{\underset{\lambda \in \Lambda_{n_{1}} \cap \Lambda_{n_{0}}}{\lambda \longrightarrow 0}}{\lim}(a-\lambda)^{-1} p_{\sigma_{n_{0}}\setminus \sigma_{n_{1}}}.$$
 Since $\underset{\underset{\lambda \in \Lambda_{n_{1}} \cap \Lambda_{n_{0}}}{\lambda \longrightarrow 0}}{\lim}(a-\lambda)^{-1} p_{\sigma_{n_{0}} \setminus \sigma_{n_{1}}} \neq 0$, we conclude that $a^{D,\sigma_{n_{0}}} \neq a^{D, \sigma_{n_{1}}}.$
\end{proof}

From the proof of the previous theorem, we see that $p_{\sigma_{n_{0}}} \neq p_{\sigma_{n_{1}}}$; which proves the non-uniqueness of idempotent $p$ involved in Theorem \ref{theok}.

\begin{rem}\rm In a private communication with the second author, Professor Sne\v{z}ana C. \v{Z}ivkovi\'c-Zlatanovi\'c pointed out that the result of \cite[Theorem 2.10]{AbZg1} is not true since it does not cover the case $\sigma(T)\setminus\sigma_b(T)=\{0\}$. We are deeply grateful to her.\\
To be clarified, we give an example inspired from \cite[Example 3.7]{Zguitti}. Let $S$ be the operator shift on $\ell^{2}(\mathbb{N})$ defined by $S(e_{n})=e_{n+1}$, where $\{ e_{n} \ : \ n \geq 1 \}$ is the usual basis of $\ell^{2}(\mathbb{N})$. We denote the identity operator of $\ell^{2}(\mathbb{N})$ by $I_{\ell^{2}(\mathbb{N})}$, and we define the operator $B$ on $H=\ell^{2}(\mathbb{N}) \oplus \ell^{2}(\mathbb{N})$ as follows
$$B=(-e_{1} \otimes e_{1} ) \oplus (\frac{1}{2}S+I_{\ell^{2}(\mathbb{N})}).$$
By virtue of \cite[Example 3.7]{Zguitti} we have $$\sigma(B)=\{-1,0\} \cup \overline{D}(1,\frac{1}{2})\mbox{ and }\sigma_{b}(B)=\{0\} \cup \overline{D}(1,\frac{1}{2}).$$ Now, we denote  the identity operator of $H$  by $I_{H}$, and we consider $T=I_{H}+B$. By the spectral mapping theorem ( see \cite[Theorem 2.1]{KK}) we have $$\sigma(T)=\{0,1\} \cup \overline{D}(2,\frac{1}{2})\mbox{ and }\sigma_{b}(T)=\{1\} \cup \overline{D}(2,\frac{1}{2}).$$
Hence $0$ is the only Riesz point of $T$, therefore the only spectral set that consists of Riesz points and $0$ is $\sigma_{0}=\{0\}$. Thus, $T$ has a unique generalized Drazin-Riesz inverse which is $S=(I_{\ell^{2}(\mathbb{N})}-e_{1} \otimes e_{1} ) \oplus (\frac{1}{2}S+2I_{\ell^{2}(\mathbb{N})})^{-1}$; while $\sigma(T) \neq \sigma_{b}(T)$.
\end{rem}

In the case $\mathcal{A}=\mathcal{L}(X)$, the following theorem gives a revised version of \cite[Theorem 2.10]{AbZg1}.

\begin{thm} \label{DRuniqueness}
Let $a \in \mathcal{A}$ be generalized Drazin-Riesz invertible such that $0 \in \sigma(a)$. If $a$ has a unique generalized Drazin-Riesz inverse  then $\sigma(a)=\sigma_{b}(a) \cup \{0\}$ and the generalized Drazin-Riesz inverse is the generalized Drazin inverse.
\end{thm}
\begin{proof}
Assume that $a$ has a unique generalized Drazin-Riesz inverse. We suppose that $\sigma(a)\neq \sigma_{b}(a)\cup \{0\}$. Then $p_{00}(a) \neq \emptyset$.
\smallskip

\noindent{\it Case 1.} $p_{00}(a)\cup \{0\}$ is finite: then we have card$(p_{00}(a)\cup \{0\}) \geq 2$. Since $0 \notin  acc \, \sigma_{b}(a)$ by Theorem \ref{theok}, $p_{00}(a) \cup \{0\}$ and $\sigma_{b}(a) \setminus \{0\}$ are disjoint clopen sets in $\sigma(a)$ such that
$$\sigma(a)=(\sigma_{b}(a)\setminus \{0\}) \cup (p_{00}(a) \cup \{0\}).$$
Now for $\sigma_{n}=p_{00}(a) \cup \{0\}=\{0,\lambda_{1},\lambda_{2},...,\lambda_{n}\}$, and $\sigma_{n}'=(\sigma_{b}(a) \setminus \{0\})$, we obtain that $a=ap_{\sigma_{n}}+ap_{\sigma_{n}'}$, where $p_{\sigma'_{n}}=1-p_{\sigma_{n}}$; and the element $s_{1}=(a-p_{\sigma_{n}})^{-1}(1-p_{\sigma_{n}})$ is a generalized Drazin-Riesz inverse of $a$.\\
Again for $\sigma_{n-1}=\{0,\lambda_{1},...,\lambda_{n-1}\}$ and $\sigma_{n-1}'=(\sigma_{b}(a) \setminus \{0\}) \cup \{ \lambda_{n} \}$, we get $s_{2}=(a-p_{\sigma_{n-1}})^{-1}(1-p_{\sigma_{n-1}})$ is also a generalized Drazin-Riesz inverse of $a$. \\
 Now we prove that $s_{1} \neq s_{2}$. Indeed,
\begin{align*}
s_{1} & = (a-p_{\sigma_{n}})^{-1}(1-p_{\sigma_{n}}) \\
      & = \left((a-p_{\sigma_{n-1}}-p_{\{\lambda_{n}\}})(1-p_{\{\lambda_{n}\}})+(a-p_{\sigma_{n}})p_{\{\lambda_{n}\}}\right)^{-1}(1-p_{\sigma_{n}}) \\
      & =\left((a-p_{\sigma_{n-1}})(1-p_{\{\lambda_{n}\}})+(a-p_{\sigma_{n}})p_{\{\lambda_{n}\}}\right)^{-1}(1-p_{\sigma_{n}})
\end{align*}
As $(a-p_{\sigma_{n-1}})$, $(a-p_{\sigma_{n}})$ and $p_{\{\lambda_{n}\}}$ are mutually commuting; and $(a-p_{\sigma_{n-1}})$ and $(a-p_{\sigma_{n}})$ are invertible in $\mathcal{A}$, we conclude by \cite[Lemma 2.1]{Djor} that
$$s_{1}=[(a-p_{\sigma_{n-1}})^{-1}(1-p_{\{\lambda_{n}\}})+(a-p_{\sigma_{n}})^{-1}p_{\{\lambda_{n}\}}](1-p_{\sigma_{n}}).$$
Also, we notice that $p_{\sigma_{n}}=p_{\sigma_{n-1}}+p_{\{\lambda_{n}\}}$ and $p_{\sigma_{n-1}}p_{\{\lambda_{n}\}}=p_{\{\lambda_{n}\}}p_{\sigma_{n-1}}=0$, hence, we conclude immediately that
 $p_{\{\lambda_{n}\}}(1-p_{\sigma_{n}})=(1-p_{\sigma_{n}})p_{\{\lambda_{n}\}}=0$ and $(1-p_{\{\lambda_{n}\}})(1-p_{\sigma_{n}})=(1-p_{\sigma_{n}})(1-p_{\{\lambda_{n}\}})=1-p_{\sigma_{n}}$. Thus
\begin{align*}
s_{1} & = (a-p_{\sigma_{n-1}})^{-1}(1-p_{\sigma_{n}}) \\
      & = (a-p_{\sigma_{n-1}})^{-1}(1-p_{\sigma_{n-1}}-p_{\{\lambda_{n}\}}) \\
      & = (a-p_{\sigma_{n-1}})^{-1}(1-p_{\sigma_{n-1}})- (a-p_{\sigma_{n-1}})^{-1}p_{\{\lambda_{n}\}} \\
      & = s_{2}- (a-p_{\sigma_{n-1}})^{-1}p_{\{\lambda_{n}\}}.
\end{align*}
As $(a-p_{\sigma_{n-1}})^{-1}$ is invertible in $\mathcal{A}$ and $p_{\{\lambda_{n}\}} \neq 0$, then $(a-p_{\sigma_{n-1}})^{-1}p_{\{\lambda_{n}\}} \neq 0.$ Therefore $s_{1}\neq s_{2}$, and this is a contradiction.
\smallskip

\noindent{\it Case 2.} $p_{00}(a)$ is infinite: If $0 \in  acc \, p_{00}(a)$, then $0 \in acc \, \sigma(a)$. Since $a$ is generalized Drazin-Riesz invertible, we conclude that the generalized Drazin-Riesz inverse of $a$ is not unique by virtue of Theorem \ref{non-uniquenessDRinv}, which contradicts our assumption. Now if $0 \notin  acc \, p_{00}(a)$, set $\sigma_{n}=\{0,\lambda_{1},\lambda_{2},...,\lambda_{n}\}$, where $(\lambda_{i})$ are non-zero Riesz points; and $\sigma_{n}'=(\sigma_{b}(a) \setminus \{0\}) \cup ((p_{00}(a) \cup \{0\})\setminus \sigma_{n})$. Hence as in the case where $p_{00}(a)$ is finite, we obtain distinct generalized Drazin-Riesz inverses for $a$, which yields again to a contradiction.
\smallskip

Therefore $\sigma(a)=\sigma_{b}(a) \cup \{0\}$. Since $a$ is generalized Drazin-Riesz invertible, $0\notin  acc \, \sigma_{b}(a)$ ($= acc \, \sigma(a)$) and so $a$ is generalized Drazin invertible.
\end{proof}

\begin{rem}\rm
Suppose that $\sigma(a)=\sigma_{b}(a) \cup \{0\}$. As $a$ is generalized Drazin-Riesz invertible, then using Theorem \ref{theok}, we get $0 \notin  acc \, \sigma_{b}(a)$ ($=acc \, \sigma(a)$). Thus, $a$ is generalized Drazin invertible, as $p_{00}(a) \subset \{0\}$, we can not have spectral sets $\sigma$ composed of $0$ and a finite family of Riesz points of $a$ except $\sigma_{0}=\{0\}$. Hence, the only generalized Drazin-Riesz inverse of the form $a^{D,\sigma}$ is $a^{D,\sigma_{0}}$ which is the generalized Drazin inverse.
\end{rem}

\begin{prop} \label{PropDrazidemp}
Let $a$ be a generalized Drazin-Riesz invertible element in $\mathcal{A}$. Let $q\in\mathcal{A}$ be an idempotent which commutes with $a$, then $aq$ and $a(1-q)$ are generalized Drazin-Riesz invertible in $\mathcal{A}$.
\end{prop}
\begin{proof}
If $0 \notin acc( \sigma(a))$, we get the result by  \cite[Lemma 1]{Mozic}. So assume that $0 \in acc\, \sigma(a)$. Then there exists $n$ such that $r_{n}< \frac{1}{2}$, and by Theorem \ref{expinv}, $a^{D,\sigma_{n}}$ is a generalized Drazin-Riesz inverse for $a$ in $\mathcal{A}$. Also $p_{\sigma_{n}}=1-aa^{D,\sigma_{n}}$ is the spectral idempotent of $a$ among $\sigma_{n}$, hence we conclude that $a^{D,\sigma_{n}}q=qa^{D,\sigma_{n}}$. Thus $aq$ commutes with $a^{D,\sigma_{n}}q$ and
$$a^{D,\sigma_{n}}qaqa^{D,\sigma_{n}}q=a^{D,\sigma_{n}}aa^{D,\sigma_{n}}q=a^{D,\sigma_{n}}q.$$
As $a(1-aa^{D,\sigma_{n}})=ap_{\sigma_{n}}$ commutes with $q$, we obtain
$$r(\pi(p_{\sigma_{n}}aq))\leq r(\pi(p_{\sigma_{n}}a))r(\pi(q))=0.$$
Hence $aq-aqa^{D,\sigma_{n}}aq$ is a Riesz element in $\mathcal{A}$.\\
 Therefore, $aq$ is generalized Drazin-Riesz invertible in $\mathcal{A}$ with a generalized Drazin-Riesz inverse $$b_{q}=a^{D,\sigma_{n}}q.$$
 To show that $a(q-1)$ is generalized Drazin-Riesz invertible, it suffices to apply the first case to the idempotent $k=(1-q)$. Clearly, $k$ commutes with the generalized Drazin-Riesz invertible element $a$. Hence $ak=a(1-q)$ is generalized Drazin-Riesz invertible.
\end{proof}

The next Theorem is a generalization of  \cite[Proposition 2.8]{KarmTaj}.
\begin{thm} \label{stabilityRieszaddition}
Let $a$ and $b \in {\mathcal A}$ be generalized Drazin-Riesz invertible elements with $ab=ba=0$. Then the sum $a+b$ is generalized Drazin-Riesz invertible.
\end{thm}
\begin{proof}
Assume that $a$ and $b$ are generalized Drazin-Riesz invertible elements in ${\mathcal A}$. Then there exist $a',b' \in {\mathcal A}$ such that $$aa'=a'a,\, a'aa'=a', a-a^2a'\mbox{ is Riesz},\, bb'=b'b, b'bb'=b', \mbox{ and }b-b^2b'\mbox{ is Riesz}.$$
Then $a'+b'$ is a generalized Drazin-Riesz inverse for $a+b$. Indeed, since $ab=ba=0$, we have $ab'=abb'^{2}=b'^{2}ba=b'a=0$, $ba'=baa'^{2}=a'^{2}ab=a'b=0$, and $b'a'=b'^{2}baa'^{2}=a'^{2}abb'^{2}=a'b'=0$. Then
$$(a+b)(a'+b')=(a'+b')(a+b),$$
and $$(a+b)(b'+a')(b'+a')=(a+b)(b'^{2}+b'a'+a'b'+a'^{2})=aa'^{2}+bb'^{2}=a'+b'.$$
Also
\begin{align*}
(a+b)-(a+b)^2(a'+b')&=(a+b)-(a^{2}+ab+ba+b^{2})(a'+b') \\
                     &=(a+b)-(a^{2}+b^{2})(a'+b')\\
                     &=a-a^{2}a'+b-b^{2}b'.
\end{align*}
Since $a-a^{2}a'$ and $b-b^{2}b'$ are commuting Riesz elements, then using Theorem \ref{TheoRieszstability}, we obtain that $(a+b)-(a+b)^2(a'+b')$ is Riesz. Therefore, $a'+b'$ is a generalized Drazin-Riesz inverse for $a+b$.
\end{proof}
\begin{cor}
Let $a\in\mathcal{A}$ and let $q\in\mathcal{A}$ be an idempotent which commutes with $a$. Then $a$ is generalized Drazin-Riesz invertible in $\mathcal{A}$ if and only if $aq$ and $a(1-q)$ are generalized Drazin-Riesz invertible in $\mathcal{A}$.
\end{cor}
\begin{proof}
$\Rightarrow )$ We apply directly Proposition \ref{PropDrazidemp}.

$\Leftarrow)$ We have $a=aq+a(1-q)$ and $(aq)(a(1-q))=(a(1-q))(aq)=0$, with $aq$ and $a(1-q)$ are generalized Drazin-Riesz invertible in $\mathcal{A}$, thus, by Theorem \ref{stabilityRieszaddition}, $a$ is generalized Drazin-Riesz invertible in $\mathcal{A}$.
\end{proof}

The {\it generalized Drazin-Riesz spectrum} of $a$ is defined by $$\sigma_{DR}(a)=\{ \lambda \in \mathbb{C} :  \lambda -a\mbox{ is not generalized Drazin-Riesz invertible} \}.$$

\begin{cor}Let $a\in\mathcal{A}$ and let $p\in\mathcal{A}$ be an idempotent which commutes with $a$. Then
$$\sigma_{DR}(a)=\sigma_{DR}(a(1-p))\cup \sigma_{DR}(ap).$$
\end{cor}

For a subset $K$ of $\mathbb{C}$ and $s \in \mathbb{N}$, we set $(K)^{s}=\{ \lambda^{s} \ : \ \lambda \in K \}$.
\begin{prop}
Let $a \in \mathcal{A}$ be generalized Drazin-Riesz invertible, than for $s \in \mathbb{N}$ and for an appropriate $n_{0} \in \mathbb{N}$ we have:
\begin{enumerate}
\item[(i)] $(a^{D,\sigma_{n_{0}}})^{D}$ exists, and $(a^{D,\sigma_{n_{0}}})^{D}=a^{2}a^{D,\sigma_{n_{0}}}$;
\item[(ii)] $a^{D,\sigma_{n_{0}}}(a^{D,\sigma_{n_{0}}})^{D}=aa^{D,\sigma_{n_{0}}}=1-p_{\sigma_{n_{0}}}$;
\item[(iii)] $((a^{D,\sigma_{n_{0}}})^{D})^{D}=a^{D,\sigma_{n_{0}}}$.
\end{enumerate}
\end{prop}
\begin{proof}
It is a special case of \cite[Theorem 2.7]{KolgsDr}.
\end{proof}

\section{Generalized Drazin-Riesz invertible elements form a regularity }
\noindent

Following M\"{u}ller \cite{Mul}, a non-empty subset $R$ of ${\mathcal A}$ is called a {\it regularity} if it satisfies the following conditions:
\begin{enumerate}
\item[(1)] if $a \in \mathcal{\mathcal A}$ and $n \in \mathbb{N}$, then  $a \in R$ if and only if $a^{n} \in R$;
\item[(2)] if $a,b,c,d$ are mutually commuting elements of $\mathcal{\mathcal A}$ and $ac+bd=1$, then $ab \in R$ if and only if $a \in R$ and $b \in R$.
\end{enumerate}

A regularity $R$ assigns to each $a \in \mathcal{A}$ a subset of $\mathbb{C}$ defined by
$$
\sigma_{R}(a)=\{\lambda \in \mathbb{C}: \lambda -a \notin R\}
$$
and called the spectrum of $a$ corresponding to the regularity $R .$ Since $Inv(\mathcal{A})\subseteq R$, then $\sigma_{R}(a) \subseteq \sigma(a)$. In general, $\sigma_{R}(a)$ is neither compact nor non-empty. However, $f(\sigma_{R}(a))=\sigma_{R}(f(a))$ for
every analytic function $f$ on a neighborhood of $\sigma(a)$ which is non-constant on each component of its domain of definition (see \cite{Mul}).
\smallskip

We denote by ${\mathcal A}^{DR}$ the class of generalized Drazin-Riesz invertible elements in ${\mathcal A}$. The set ${\mathcal A}^{DR}$ is not empty since it contains at least the element $0$. Moreover, the class of generalized Drazin-Riesz invertible elements in a Banach algebra is not necessarily open as shown in the following example.
\begin{ex}\rm
 Let $\mathcal{A}=(C([0,1]),\| . \|_{\infty})$ be the Banach algebra of all continuous complex valued functions on $[0,1]$. For each $n$, let $a_n \in \mathcal{A}$ be defined by $a_n(t)={1\over n} t$. Then we have $\sigma(a_n)= a_n([0,1])=[0,{1\over n}]= acc \, \sigma(a_n)$ and so $\sigma_{b}(a_n)= [0,{1\over n}]$. Thus $0 \in acc\, \sigma_{b}(a_n)$. Finally, $a_n$ is not generalized Drazin-Riesz invertible, while $0$ is generalized Drazin-Riesz invertible. We conclude that the class of generalized Drazin-Riesz invertible elements is not open in $\mathcal{A}$.
\end{ex}

\begin{lem} \label{Lemmakeyregularity}
Let $a,b,v$ and $w \in \mathcal{A}$ be mutually commuting elements such that $av+bw=1$. If $ab$ is generalized Drazin-Riesz invertible then $a$ and $b$ are generalized Drazin-Riesz invertible.
\end{lem}
\begin{proof}
Assume that $ab$ is generalized Drazin-Riesz invertible. If $0 \notin acc\, \sigma(ab)$, then $ab$ is generalized Drazin invertible and by \cite[Lemma 1.1]{Lub}, $a$ and $b$ are generalized Drazin invertible. In particular, they are generalized Drazin-Riesz invertible.\\
So suppose that $0 \in acc \, \sigma(ab)$. Theorem \ref{expinv} implies that there exists a positive integer $n_{0}$ such that $(ab)^{D,\sigma_{n_{0}}}$ is a generalized Drazin-Riesz inverse of $ab$ and  $p_{\sigma_{n_{0}}}=1-(ab)(ab)^{D,\sigma_{n_{0}}}$ is the spectral idempotent of $ab$ related to $\sigma_{n_{0}}$.
\smallskip

\noindent \textbf{Claim 1.} $b(ab)^{D,\sigma_{n_{0}}}$ is a generalized Drazin-Riesz inverse for $a(1-p_{\sigma_{n_{0}}})$:
\smallskip

\noindent Set $u=b(ab)^{D,\sigma_{n_{0}}}$. As $ab$ commutes with $u$, then $p_{\sigma_{n_{0}}}$ and $u$ commute. Also $a$ commutes with $(ab)^{D,\sigma_{n_{0}}}$ and with $b$, we conclude that $u$ commutes with $a(1-p_{\sigma_{n_{0}}})$. We point out that
$u(1-p_{\sigma_{n_{0}}})=u$ as $(ab)^{D,\sigma_{n_{0}}}p_{\sigma_{n_{0}}}=0$. We get, \\
$$ua=au=ab(ab)^{D,\sigma_{n_{0}}}=1-p_{\sigma_{n_{0}}}$$ and $$ua(1-p_{\sigma_{n_{0}}})u=(1-p_{\sigma_{n_{0}}})^{2}u=u.$$
Further,
$$a(1-p_{\sigma_{n_{0}}})-a(1-p_{\sigma_{n_{0}}})ua(1-p_{\sigma_{n_{0}}})=a(1-p_{\sigma_{n_{0}}})-a(1-p_{\sigma_{n_{0}}})^{3}=0.$$
Therefore, $a(1-p_{\sigma_{n_{0}}})$ is generalized Drazin-Riesz invertible and $u$ is a generalized Drazin-Riesz inverse for $a(1-p_{\sigma_{n_{0}}})$.\\
We notice also that
$$1-ua(1-p_{\sigma_{n_{0}}})=p_{\sigma_{n_{0}}}.$$
Hence $p_{\sigma_{n_{0}}}$ coincides with the idempotent associated with the generalized Drazin-Riesz inverse $u$ of $a(1-p_{\sigma_{n_{0}}})$.
\smallskip

\noindent \textbf{Claim 2.} $avp_{\sigma_{n_{0}}}$ is generalized Drazin-Riesz invertible:
\smallskip

\noindent We have
\begin{align*}
avp_{\sigma_{n_{0}}}-(avp_{\sigma_{n_{0}}})^{2}& = av (1-av)p_{\sigma_{n_{0}}} \\
                                               & = avbwp_{\sigma_{n_{0}}}=(abp_{\sigma_{n_{0}}})vw.
\end{align*}
Since $abp_{\sigma_{n_{0}}}$ is Riesz and commutes with $vw$, then by Theorem \ref{TheoRieszstability}, we obtain that $(abp_{\sigma_{n_{0}}})vw$ is Riesz. Thus, $avp_{\sigma_{n_{0}}}-(avp_{\sigma_{n_{0}}})^{2}$ is Riesz. \\
Also, by considering the analytic function $f$ on $\mathbb{C}$ defined by $f(\lambda)=\lambda-\lambda^{2}$, we have
\begin{align*}f(\sigma(\pi(avp_{\sigma_{n_{0}}}))) &=\sigma(f(\pi(avp_{\sigma_{n_{0}}}))) \\
         & = \sigma(\pi(avp_{\sigma_{n_{0}}}-(avp_{\sigma_{n_{0}}})^{2}))=\{0\}.
\end{align*}
Then $\sigma(\pi(avp_{\sigma_{n_{0}}})) \subset \{0,1\}$ and so $0 \in iso \ \sigma_{e}(avp_{\sigma_{n_{0}}})\cap iso \sigma_{e}(1-avp_{\sigma_{n_{0}}})$. Thus $0 \notin acc \,  \sigma_{e}(avp_{\sigma_{n_{0}}})\cap acc\,\sigma_{e}(1-avp_{\sigma_{n_{0}}})$.\\ Let $c\in \{1-avp_{\sigma_{n_{0}}}, avp_{\sigma_{n_{0}}} \}$. We will show that $0 \notin  acc ( acc \, \sigma(c))$.
 By the sake of contradiction assume that $0 \in  acc \, (acc \, \sigma(c))$. There exists $(\lambda_{i})_{i} \subset  acc \, \sigma(c)$ such that $\lambda_{i} \longrightarrow 0, \mbox{ as } i \longrightarrow \infty$. Thus for all $i \in \mathbb{N}$ there exists a sequence $(s_{i,k})_{k} \subset \sigma(c)$ such that $s_{i,k} \longrightarrow \lambda_{i}, \mbox{ as } k \longrightarrow \infty$. Hence $f(s_{i,k}) \longrightarrow f(\lambda_{i}), \mbox{ as } k \longrightarrow \infty$, and using the spectral mapping theorem, we get
$(f(\lambda_{i}))_{i} \subset  acc \, \sigma(f(c))$. Then $f(\lambda_{i}) \longrightarrow 0, \mbox{ as } i \longrightarrow \infty$, therefore
$$0 \in  acc \, (acc \, \sigma(f(c))) \subset acc \, \sigma_{b}(c-c^{2})=  acc \, \sigma_{b}(avp_{\sigma_{n_{0}}}-(avp_{\sigma_{n_{0}}})^{2}),$$
which is a contradiction.
Consequently, $0 \notin  acc \, \sigma_{b}(c)$. By virtue of Theorem \ref{theok} we conclude that $avp_{\sigma_{n_{0}}}$ and $1-avp_{\sigma_{n_{0}}}$ are generalized Drazin-Riesz invertible.
\smallskip

\noindent \textbf{Claim 3.} $ap_{\sigma_{n_{0}}}$ is generalized Drazin-Riesz invertible:
\smallskip

\noindent Let $\varpi_{k}$ be the spectral set associated with the generalized Drazin-Riesz invertible element $avp_{\sigma_{n_{0}}}$. We show that  $q=vp_{\sigma_{n_{0}}}(avp_{\sigma_{n_{0}}})^{D,\varpi_{k}}$ is a generalized Drazin-Riesz inverse of $ap_{\sigma_{n_{0}}}$, where $\varpi_{k}$ is a spectral set composed of $0$ union the sequence  of non-zero Riesz points of $avp_{\sigma_{n_{0}}}$.
\smallskip

We have $q$ commutes with $ap_{\sigma_{n_{0}}}$, while
$$qap_{\sigma_{n_{0}}}q=vp_{\sigma_{n_{0}}}(avp_{\sigma_{n_{0}}})^{D,\varpi_{k}}ap_{\sigma_{n_{0}}}vp_{\sigma_{n_{0}}}(avp_{\sigma_{n_{0}}})^{D,\varpi_{k}}=vp_{\sigma_{n_{0}}}(avp_{\sigma_{n_{0}}})^{D,\varpi_{k}}=q.$$
And
\begin{align}
ap_{\sigma_{n_{0}}}-ap_{\sigma_{n_{0}}}qap_{\sigma_{n_{0}}}&=ap_{\sigma_{n_{0}}}(1-vp_{\sigma_{n_{0}}}(avp_{\sigma_{n_{0}}})^{D,\varpi_{k}}ap_{\sigma_{n_{0}}}) \label{eq3.1} \\
& = ap_{\sigma_{n_{0}}}(1-avp_{\sigma_{n_{0}}}(avp_{\sigma_{n_{0}}})^{D,\varpi_{k}}) \\
& = a(av+bw)p_{\sigma_{n_{0}}}(1-avp_{\sigma_{n_{0}}}(avp_{\sigma_{n_{0}}})^{D,\varpi_{k}}) \\
& = a(avp_{\sigma_{n_{0}}})(1-avp_{\sigma_{n_{0}}}(avp_{\sigma_{n_{0}}})^{D,\varpi_{k}})) \\
& + abp_{\sigma_{n_{0}}}(w(1-avp_{\sigma_{n_{0}}}(avp_{\sigma_{n_{0}}})^{D,\varpi_{k}})). \label{eq3.5}
\end{align}
Notice that all elements of (\ref{eq3.1})-(\ref{eq3.5}) commute, and $1-avp_{\sigma_{n_{0}}}(avp_{\sigma_{n_{0}}})^{D,\varpi_{k}}$ is the spectral idempotent of $avp_{\sigma_{n_{0}}}$ related to $\varpi_{k}$.
\smallskip

 According to Theorem \ref{TheoRieszstability} and taking into account that $(avp_{\sigma_{n_{0}}})(1-avp_{\sigma_{n_{0}}}(avp_{\sigma_{n_{0}}})^{D,\varpi_{k}}))$ and $abp_{\sigma_{n_{0}}}$ are Riesz elements, we obtain that $ a(avp_{\sigma_{n_{0}}})(1-avp_{\sigma_{n_{0}}}(avp_{\sigma_{n_{0}}})^{D,\varpi_{k}}))$ and
 $abp_{\sigma_{n_{0}}}(w(1-avp_{\sigma_{n_{0}}}(avp_{\sigma_{n_{0}}})^{D,\varpi_{k}}))$  are both Riesz elements. Again according to Theorem \ref{TheoRieszstability}, their sum is also a Riesz element. Hence $ap_{\sigma_{n_{0}}}-ap_{\sigma_{n_{0}}}qap_{\sigma_{n_{0}}}$ is a Riesz element.
 Finally, $q$ is a generalized Drazin-Riesz inverse of $ap_{\sigma_{n_{0}}}$.
\smallskip

\noindent \textbf{Claim 4.} $a$ is generalized Drazin-Riesz invertible:
\smallskip

\noindent We have shown that $a(1-p_{\sigma_{n_{0}}})$ and $ap_{\sigma_{n_{0}}}$ are generalized Drazin-Riesz invertible. Since they commute and their multiplication  is equal to 0, we use Theorem \ref{stabilityRieszaddition} to conclude that $a=a(1-p_{\sigma_{n_{0}}})+ap_{\sigma_{n_{0}}}$ is generalized Drazin-Riesz invertible.
\end{proof}
\begin{thm} \label{regularity}
The class $\mathcal{A}^{DR}$ of all generalized Drazin-Riesz invertible elements forms a regularity.
\end{thm}
\begin{proof}
1) If $a \in \mathcal{A}^{DR}$, it is easy to see that $a^{n}$ is also generalized Drazin-Riesz invertible. Conversely, let $n\in\mathbb{N}$ and
assume that $a^{n} \in \mathcal{A}^{DR}$. Then by Theorem \ref{theok}, $0 \notin  acc \, \sigma_{b}(a^{n})= acc(\sigma_{e}(a^n) \cup   acc (\sigma(a^{n})))$.
  Since, $\sigma_{e}(a^{n})=\{ \lambda^{n} \ : \ \lambda \in \sigma_{e}(a) \}$, and $\sigma(a^{n})=\{\lambda^{n} \ : \ \lambda \in \sigma(a) \}$, we conclude that $0 \notin  acc \, \sigma_{b}(a)= acc \, (\sigma_{e}(a) \cup acc \, \sigma(a) )$.  Again by Theorem \ref{theok}, we deduce that $a$ is generalized Drazin-Riesz invertible. \\
2) Let $a,b,v$ and $w \in \mathcal{A}$ be mutually commuting elements such that $av+bw=1$. If $ab$ is generalized Drazin-Riesz invertible then by Lemma \ref{Lemmakeyregularity},  $a$ and $b$ are generalized Drazin-Riesz invertible.
\smallskip

Conversely, assume that $a$ and $b$ are generalized Drazin-Riesz invertible elements in $\mathcal{A}$. Set $$p=1-(1-p_{\sigma_{n_{0}},a})(1-p_{\sigma_{n_{0}},b}),$$ where $p_{\sigma_{n_{0}},a}$ (resp. $p_{\sigma_{n_{0}},b}$) is the spectral idempotent associated with the spectral set $\sigma_{n_{0},a}$ (resp. $\sigma_{n_{0},b}$) of $a$ (resp. $b$).\\
Then $$a+p_{\sigma_{n_{0},a}}\mbox{ and }b+p_{n_{0},b}\mbox{ are Browder elements in  }\mathcal{A},$$ and $$ap_{\sigma_{n_{0},a}}\mbox{ and }bp_{\sigma_{n_{0},b}}\mbox{ are Riesz elements in }\mathcal{A}.$$
Using Lemma \ref{lem1}, $(a+p_{\sigma_{n_{0},a}})(b+p_{\sigma_{n_{0},b}})$ is Browder. Hence, $(a+p_{\sigma_{n_{0},a}})(b+p_{\sigma_{n_{0},b}})(1-p)=ab(1-p)$ is also a Browder element in $(1-p)\mathcal{A}(1-p)$.\\
Since $1-p=p_{\sigma_{n_{0},a}}+p_{\sigma_{n_{0},b}}-p_{\sigma_{n_{0},a}}p_{\sigma_{n_{0},b}}$, then $$abp=ap_{\sigma_{n_{0},a}}b+abp_{\sigma_{n_{0},b}}-ap_{\sigma_{n_{0},a}}bp_{\sigma_{n_{0},b}}.$$
  It follows from Theorem \ref{TheoRieszstability} that $abp$ is a Riesz element in $\mathcal{A}$.\\
On the other hand, $$ab+p=abp+p+ab(1-p)=(1+abp)p+ab(1-p).$$ We have shown that $ab(1-p)$ is a Browder element in $(1-p)\mathcal{A}(1-p)$. Using  \cite[Corollary 4.13]{Pearlman} $(1+abp)p$ is also a Browder element in $p\mathcal{A}p$ and so $ab+p$ is a Browder element in $\mathcal{A}$. . Therefore we conclude that $ab \in \mathcal{A}^{DR}$.
\end{proof}

The spectral mapping theorem for the spectrum relative to the regularity $\mathcal{A}^{DR}$ denoted by $\sigma_{DR}(a)$ is an immediate consequence of Theorem \ref{regularity} and  \cite[Theorem I.6.7]{Mul}.
\begin{thm} \label{SpectralmappingtheoforDR}
Let $a \in \mathcal{A}$. If $f$ is any function holomorphic in an open neighborhood of  $\sigma(a)$ and non-constant on any component of $\sigma(a)$, then
$$f(\sigma_{DR}(a))=\sigma_{DR}(f(a)).$$
\end{thm}

\section{Further results for generalized Drazin-Riesz invertible elements}

This section is devoted to generalize various results in \cite{Djor} to the class of generalized Drazin-Riesz invertible elements.
\smallskip
\begin{thm} \label{Theo6.1}
Let $a \in \mathcal{A}$ be generalized Drazin-Riesz invertible, and $a^{D,\sigma_{n}}$ be a generalized Drazin-Riesz inverse of $a$ for a large enough $n$, then
$$(a^{D,\sigma_{n}})^{s}=(a^{s})^{D,(\sigma_{n})^{s}}$$ for all  $s\in\mathbb{N}$.
\end{thm}
\begin{proof}

 Let $a$ be generalized Drazin-Riesz invertible, with $a^{D,\sigma_{n}}$ one of its generalized Drazin-Riesz inverses for a large enough $n$. Then $0 \notin \sigma_{DR}(a)$, using Theorem \ref{SpectralmappingtheoforDR}, we get  $0 \notin \sigma_{DR}(a^{s}),$ for all $s\in \mathbb{N}$. Therefore $a^{s}$ is generalized Drazin-Riesz invertible.\\
We have $(\sigma(a))^{s}=\sigma(a^{s})=(\sigma_{n}')^{s} \cup (\sigma_{n})^{s}$. Indeed, for $n$ large enough we have $\underset{\lambda \in \sigma_{n}}{\sup} |\lambda|< \frac{1}{2}$, and $0 \in iso \, (\sigma'_{n} \cup \{0\})$, this means that there exists $\eta > 0$ such that $(D(0,\eta)\setminus \{0\} ) \cap \sigma'_{n}= \emptyset$. Hence $\underset{\lambda \in \sigma'_{n}}{\inf} |\lambda| > \eta$, we can choose $n$ such that $\underset{\lambda \in \sigma_{n}}{\sup} |\lambda|< \min(\frac{1}{2},\eta)$ and $\underset{\lambda \in \sigma'_{n}}{\inf} |\lambda| > \min(\frac{1}{2},\eta)$. Thus, there is no power $s$ and no elements $\lambda \in \sigma'_{n}$ and $\gamma \in \sigma_{n}$ such that $(\lambda)^{s}=(\gamma)^{s}$. Therefore, $$(\sigma'_{n})^{s} \cap (\sigma_{n})^{s}= \emptyset.$$
Now the result follows at once from \cite[Theorem 2.7(i)]{KolgsDr}.
\end{proof}

\begin{thm} \label{Theo1.1Djordjevic}
Let $a,b \in \mathcal{A}$, $ab=ba$, $0 \notin  acc \, \sigma_{b}(a) \cup  acc \, \sigma_{b}(b)$, $0 \in (acc \, \sigma(a)) \cap  (acc \, \sigma(b))$, and $(a+\lambda b)^{-1}$ exists for some $\lambda \in \mathbb{C}$. Then, there exists $n_{0} \in \mathbb{N}$ such that $\sigma_{n_{0},a}$ and $\sigma_{n_{0},b}$ are spectral sets of $\sigma(a)$ and $\sigma(b)$ respectively, and
$$(1-aa^{D,\sigma_{n_{0},a}})bb^{D,\sigma_{n_{0},b}}=(1-aa^{D,\sigma_{n_{0},a}}).$$
\end{thm}
\begin{proof} Without loss of generality we assume that $(\lambda b + a)$ is invertible for some $\lambda \in \mathbb{C}\setminus ( \sigma_{n_{0},a} \cup \sigma_{n_{0},b})$. Also form equality $(\lambda b)^{D, \sigma_{n_{0},b}}= \lambda^{-1} b^{D, \sigma_{n_{0},b}}$, we may assume that $\lambda = -1$, i.e.,  $(a-b)$ is invertible.
We have $$p_{\sigma_{n_{0},a}}=\frac{1}{2 \pi i} \int_{\Gamma_{n_{0}}}(z-a)^{-1}dz\mbox{ and }p_{\sigma_{n_{0},b}}= \frac{1}{2 \pi i} \int_{\Gamma_{n_{0}}}(u-b)^{-1}du,$$
for an accommodate countour $\Gamma_{n_{0}}$ around $\sigma_{n_{0},a} \cup \sigma_{n_{0},b}$.
Since $p_{\sigma_{n_{0},a}}=1-aa^{D,\sigma_{n_{0},a}}$ and $p_{\sigma_{n_{0},b}}=1-bb^{D,\sigma_{n_{0},b}}$, then it suffices to show
$$\left(\frac{1}{2 \pi i} \int_{\Gamma_{n_{0}}}(z-a)^{-1}dz\right)\left( \frac{1}{2 \pi i} \int_{\Gamma_{n_{0}}}(u-b)^{-1}du\right)=0.$$
From the equality $$(z-a)^{-1}-(u-b)^{-1}=\left( (u-z)+(a-b)\right)\left((z-a)^{-1}(u-b)^{-1}\right)$$ and the fact that $a-b$ is invertible, it follows that $(u-z)+(a-b)$ is invertible for small values of $u-z$. For this purpose, we may take $n_{0}$ such that for all $ z,u \in \Gamma_{n_{0}}, \ (u-z)+(a-b)$ is invertible.\\
For $(z,u)\in\Gamma_{n_{0}} \times \Gamma_{n_{0}}$, let $F(z,u)$ be the function defined by
$$F(z,u)=\left((u-z)+(a-b)\right)^{-1}\left((z-a)^{-1}-(u-b)^{-1}\right)=(z-a)^{-1}(u-b)^{-1}.$$
Then $F$ is continuous on the set $\Gamma_{n_{0}} \times \Gamma_{n_{0}}$. Moreover, the functions $$u \longmapsto \left((u-z)+(a-b)\right)^{-1}(z-a)^{-1}\mbox{ and }z \longmapsto \left((u-z)+(a-b)\right)^{-1}(u-b)^{-1}$$ are analytic in a neighbourhood of $\sigma_{n_{0},a}\cup \sigma_{n_{0},b}$. Consequently,
\begin{align*}
p_{\sigma_{n_{0},a}}p_{\sigma_{n_{0},b}}&=\left(\frac{1}{2\pi i}\right)^{2} \int\int_{\Gamma_{n_{0}}\times \Gamma_{n_{0}}} F(z,u) dz du\\
 &=\left(\frac{1}{2\pi i}\right)^{2} \int\int_{\Gamma_{n_{0}}\times \Gamma_{n_{0}}}\left((u-z)+(a-b)\right)^{-1}\left((z-a)^{-1}-(u-b)^{-1}\right) du dz \\
 &= \frac{1}{2 \pi i} \int_{\Gamma_{n_{0}}} \left( \frac{1}{2 \pi i} \int_{\Gamma_{n_{0}}} \left((u-z)+(a-b)\right)^{-1}(z-a)^{-1}du\right)dz \\
 & - \frac{1}{2 \pi i} \int_{\Gamma_{n_{0}}} \left( \frac{1}{2 \pi i} \int_{\Gamma_{n_{0}}} \left((u-z)+(a-b)\right)^{-1}(u-b)^{-1}dz\right)du=0.
\end{align*}
\end{proof}

\begin{thm}
Assume that the conditions of Theorem \ref{Theo1.1Djordjevic} are satisfied. If $ r(ap_{\sigma_{n_{0}},a}b^{D,\sigma_{n_{0},b}}) <(r(a^{D,\sigma_{n_{0}},a}b))^{-1}$, then for all $\lambda \in (D(0,(r(a^{D,\sigma_{n_{0},a}}b))^{-1}) \setminus \overline{D(0,r(ap_{\sigma_{n_{0}},a}b^{D,\sigma_{n_{0},b}}))})$, we have
$$(\lambda b -a )^{-1}=b^{\sigma_{n_{0},b}}(1-aa^{D,\sigma_{n_{0},a}}) \sum_{n=1}^{\infty}(ab^{D,\sigma_{n_{0},b}})^{n-1} \lambda^{-n}-a^{D,\sigma_{n_{0},a}}\sum_{n=0}^{\infty}(a^{D,\sigma_{n_{0},a}}b)\lambda^{n}.$$
\end{thm}
\begin{proof}
By virtue of Theorem \ref{Theo1.1Djordjevic}, we have $p_{\sigma_{n_{0},a}}=(1-p_{\sigma_{n_{0},b}})p_{\sigma_{n_{0},a}}=p_{\sigma_{n_{0},a}}(1-p_{\sigma_{n_{0},b}}).$
Then
\begin{align*}
(\lambda b - a) &= (\lambda b-a)p_{\sigma_{n_{0},a}}+(\lambda b - a)(1-p_{\sigma_{n_{0},a}})\\
&= (\lambda b-a)(1-p_{\sigma_{n_{0},b}})p_{\sigma_{n_{0},a}}+(\lambda b - a)(1-p_{\sigma_{n_{0},a}}) \\
&=\left(\lambda b(1-p_{\sigma_{n_{0},b}})-a\right)(1-p_{\sigma_{n_{0},b}})p_{\sigma_{n_{0},a}}+\left(\lambda b - a(1-p_{\sigma_{n_{0},a}})\right)(1-p_{\sigma_{n_{0},a}})\\
&= (\lambda - ab^{D,\sigma_{n_{0},b}})b(1-p_{\sigma_{n_{0},b}}) p_{\sigma_{n_{0},a}}+(\lambda b a^{D,\sigma_{n_{0},a}}-1)a(1-p_{\sigma_{n_{0},a}}).\\
&=(\lambda - ab^{D,\sigma_{n_{0},b}})b p_{\sigma_{n_{0},a}}+(\lambda b a^{D,\sigma_{n_{0},a}}-1)a(1-p_{\sigma_{n_{0},a}}).
\end{align*}
For $|\lambda| <(r(a^{D,\sigma_{n_{0},a}}b))^{-1}$, we have 
\begin{equation}\label{eq.6.1}
	(\lambda b - a)^{-1}(1-p_{\sigma_{n_{0},a}})=(\lambda b a^{D,\sigma_{n_{0},a}}-1)^{-1}a^{D,\sigma_{n_{0}},a}=-\sum_{k=0}^{\infty}\lambda^{k}b^{k}(a^{D,\sigma_{n_{0},a}})^{k+1}.
\end{equation}
Now if $|\lambda|> r(ap_{\sigma_{n_{0},a}}b^{D,\sigma_{n_{0},b}})$ then \begin{equation}\label{eq.6.2}
	 (\lambda-b^{D,\sigma_{n_{0},b}}ap_{\sigma_{n_{0},a}})^{-1}b^{D,\sigma_{n_{0},b}}p_{\sigma_{n_{0},a}}=\sum_{k=0}^{\infty} \lambda^{-k}a^{k}(b^{D,\sigma_{n_{0},b}})^{k+1}p_{\sigma_{n_{0},a}}.
\end{equation}
Therefore, if $ r(ap_{\sigma_{n_{0}},a}b^{D,\sigma_{n_{0},b}}) <(r(a^{D,\sigma_{n_{0}},a}b))^{-1}$ then combining (\ref{eq.6.1}) and (\ref{eq.6.2}) we obtain for all $\lambda \in (D(0,(r(a^{D,\sigma_{n_{0}},a}b))^{-1}) \setminus \overline{D(0,r(ap_{\sigma_{n_{0},a}}b^{D,\sigma_{n_{0},b}}))})$, 
$$(\lambda b -a )^{-1}=b^{\sigma_{n_{0},b}}p_{\sigma_{n_{0}}} \sum_{k=1}^{\infty}(ab^{D,\sigma_{n_{0},b}})^{k-1} \lambda^{-k}-a^{D,\sigma_{n_{0},a}}\sum_{n=0}^{\infty}\lambda^{k}(a^{D,\sigma_{n_{0},a}})^{k}b^{k}.$$
\end{proof}

Let $s \in \mathbb{N}$, for a sufficiently large $n_{0}$, we obtain $(\sigma_{n_{0}})^{s}$ on some open disk $D(0,r')$; where $0 <r' < \frac{1}{4}$ and $D(0,r') \cap (\sigma_{n_{0}}')^{s}= \emptyset$. As $(\sigma_{n_{0}})^{s} \subset D(0,r')$, we immediately get that $-(\sigma_{n_{0}})^{s} \subset D(0,r')$. In the next theorem, we denote by $\Lambda_{-(\sigma_{n_{0}})^{s}}$, the set $\Lambda_{-(\sigma_{n_{0}})^{s}}=D(0,r')\setminus (-(\sigma_{n_{0}})^{s})$.

\begin{thm} \label{Theopuissancesetlim}
Let $a \in \mathcal{A}$, $0 \in iso\, \sigma_{b}(a)$, and let $s,l,t\in\mathbb{N}$. Then
\begin{align*}
 \underset{\underset{\lambda \in \Lambda_{-(\sigma_{n_{0}})^{s}} \cap \Lambda_{n_{0}}}{\lambda \longrightarrow 0}}{\lim} ( \lambda +a^{s})^{-l}(a^{D,\sigma_{n_{0}}})^{t}=(a^{D,\sigma_{n_{0}}})^{sl+t}=(a^{sl+t})^{D,(\sigma_{n_{0}})^{sl+t}}.
 \end{align*}
\end{thm}
\begin{proof}
By Theorem \ref{Theo6.1} (i), $a^{s}$ is generalized Drazin-Riesz invertible in $\mathcal{A}$, hence $-(a)^{s}$ is generalized Drazin-Riesz invertible in $\mathcal{A}$. So according to Theorem \ref{limitexpofDRinv}, we have
\begin{align}
\label{Eq6.4(1)}  \underset{\underset{\lambda \in \Lambda_{-(\sigma_{n_{0}})^{s}} \cap \Lambda_{n_{0}}}{\lambda \longrightarrow 0}}{\lim} ( \lambda +a^{s})^{-1}(1-p_{-(\sigma_{n_{0}})^{s}})& =  \underset{\underset{\lambda \in \Lambda_{-(\sigma_{n_{0}})^{s}} \cap \Lambda_{n_{0}}}{\lambda \longrightarrow 0}}{\lim} ( \lambda -(-a^{s}))^{-1}(1-p_{-(\sigma_{n_{0}})^{s}})\\
\label{Eq6.4(2)} &=-(-a^{s})^{D,-(\sigma_{n_{0}})^{s}}.
\end{align}
First of all, we show that $(-a^{s})^{D,-(\sigma_{n_{0}})^{s}}=-(a^{D,\sigma_{n_{0}}})^{s}$. By taking $\Gamma = C(0,r')$ the circle $C(0,r')$ which is the boundary of $\bar{D}(0,r')$ describing $\Lambda_{n_{0},s}$. We know that $\Gamma$ is surrounding $(\sigma_{n_{0}})^{s}$, thus, it surrounds also $-(\sigma_{n_{0}})^{s}$. Hence we have
\begin{align*}
p_{-(\sigma_{n_{0}})^{s}}&=\frac{1}{2\pi i} \int_{\Gamma} ( \lambda + a^{s})^{-1}d\lambda \\
                         &=\frac{-1}{2 \pi i} \int_{\Gamma}(- \lambda - a^{s})^{-1} d\lambda \\
                         &= \frac{1}{2 \pi i} \int_{\Gamma}(\mu-a^{s})^{-1} d\mu \\
                         &=p_{(\sigma_{n_{0}})^{s}}.
\end{align*}
On the other hand, we have
\begin{align*}
(-a^{s})^{D,-(\sigma_{n_{0}})^{s}}&=(-a^{s}-p_{-(\sigma_{n_{0}})^{s}})^{-1}(1-p_{-(\sigma_{n_{0}})^{s}}) \\
&=(-a^{s}-p_{(\sigma_{n_{0}})^{s}})^{-1}(1-p_{(\sigma_{n_{0}})^{s}})\\
&=-(a^{s}+p_{(\sigma_{n_{0}})^{s}})^{-1}(1-p_{(\sigma_{n_{0}})^{s}})\\
&=-(a^{s}-p_{(\sigma_{n_{0}})^{s}}+2p_{(\sigma_{n_{0}})^{s}})^{-1}(1-p_{(\sigma_{n_{0}})^{s}})\\
&=-(a^{s}-p_{(\sigma_{n_{0}})^{s}})^{-1}(1-p_{(\sigma_{n_{0}})^{s}})=-(a^{s})^{D,(\sigma_{n_{0}})^{s}}.
\end{align*}
And by Lemma \ref{Theo6.1} (i), we get that
$$-(a^{s})^{D,(\sigma_{n_{0}})^{s}}=-(a^{D,\sigma_{n_{0}}})^{s}=(-a^{s})^{D,-(\sigma_{n_{0}})^{s}}.$$
Also by considering (\ref{Eq6.4(1)}) and (\ref{Eq6.4(2)}), we obtain
$$(a^{D,\sigma_{n_{0}}})^{s}=\underset{\underset{\lambda \in \Lambda_{-(\sigma_{n_{0}})^{s}} \cap \Lambda_{n_{0}}}{\lambda \longrightarrow 0}}{\lim} ( \lambda +a^{s})^{-1}(1-p_{(\sigma_{n_{0}})^{s}}).$$
Hence we have $$(a^{D,\sigma_{n_{0}}})^{sl}=\underset{\underset{\lambda \in \Lambda_{-(\sigma_{n_{0}})^{s}} \cap \Lambda_{n_{0}}}{\lambda \longrightarrow 0}}{\lim} ( \lambda +a^{s})^{-l}(1-p_{(\sigma_{n_{0}})^{s}}).$$
Also, we have by virtue of \cite[Theorem 2.7(i)]{KolgsDr}, $p_{(\sigma_{n_{0}})^{s}}=p_{\sigma_{n_{0}}}$. Therefore,
\begin{align*}
(a^{D,\sigma_{n_{0}}})^{l}=\underset{\underset{\lambda \in \Lambda_{-(\sigma_{n_{0}})^{s}} \cap \Lambda_{n_{0}}}{\lambda \rightarrow 0}}{\lim}(\lambda + a^{s})^{-l}(1-p_{\sigma_{n_{0}}}).
\end{align*}
Thus
\begin{align*}
(a^{D,\sigma_{n_{0}}})^{sl+t}&=(a^{D,\sigma_{n_{0}}})^{sl}(a^{D,\sigma_{n_{0}}})^{t}\\
                             &= \underset{\underset{\lambda \in \Lambda_{-(\sigma_{n_{0}})^{s}} \cap \Lambda_{n_{0}}}{\lambda \rightarrow 0}}{\lim}(\lambda + a^{s})^{-l}(1-p_{\sigma_{n_{0}}}) (a^{D,\sigma_{n_{0}}})^{t}\\
                             &=\underset{\underset{\lambda \in \Lambda_{-(\sigma_{n_{0}})^{s}} \cap \Lambda_{n_{0}}}{\lambda \rightarrow 0}}{\lim}(\lambda + a^{s})^{-l}(a^{D,\sigma_{n_{0}}})^{t}(1-p_{\sigma_{n_{0}}})\\
                             &= \underset{\underset{\lambda \in \Lambda_{-(\sigma_{n_{0}})^{s}} \cap \Lambda_{n_{0}}}{\lambda \rightarrow 0}}{\lim}(\lambda + a^{s})^{-l}(a^{D,\sigma_{n_{0}}})^{t}.
\end{align*}
Finally, by Lemma \ref{Theo6.1}, part (i), we obtain that
$$(a^{D,\sigma_{n_{0}}})^{sl+t}=(a^{sl+t})^{D,(\sigma_{n_{0}})^{sl+t}}.$$
Hence
$$\underset{\underset{\lambda \in \Lambda_{-(\sigma_{n_{0}})^{s}} \cap \Lambda_{n_{0}}}{\lambda \rightarrow 0}}{\lim}(\lambda + a^{s})^{-l}(a^{D,\sigma_{n_{0}}})^{t}=(a^{D,\sigma_{n_{0}}})^{sl+t}=(a^{sl+t})^{D,(\sigma_{n_{0}})^{sl+t}}$$
as desired.
\end{proof}

If $\sigma_{n_{0}}=\{0\}$, we get directly \cite[Theorem 2.3]{Djor} by applying the last theorem. We finish by the case when $\mathcal{A}$ is a Banach algebra with an involution.
\begin{prop} \label{conv}
Let $\mathcal{A}$ be a semi-simple $\mathcal{C}^{*}$-Banach algebra with unit 1. Then $a \in \mathcal{A}$ is generalized Drazin-Riesz invertible if and only if $a^{*}$ is generalized Drazin-Riesz invertible.
\end{prop}
\begin{proof}
 By Theorem \ref{theok}, $a$ is generalized Drazin-Riesz invertible if and only if there exists a commuting idempotent $p \in \mathcal{A}$  with $a$, such that $a+p$ is invertible and $ap$ is Riesz in $\mathcal{A}$.\\
This implies that $(a+p)^{*}=a^{*}+p^{*}$ is invertible and $(ap)^{*}=a^{*}p^{*}=p^{*}a^{*}$ is Riesz, where $p^{*}$ is a commuting idempotent with $a^{*}$. Indeed, $((a+p)^{-1})^{*}=((a+p)^{*})^{-1}=(a^{*}+p^{*})^{-1}$, and $p^{*}=(p^{2})^{*}=(p^{*})^{2}$, also $p^{*}a^{*}=(ap)^{*}=(pa)^{*}=a^{*}p^{*}$. On the other hand, $\sigma_{e}(a^{*}p^{*})=\{ \overline{\lambda}  :  \lambda \in \sigma_{e}(ap) \}=\{0\}$, therefore $a^{*}p^{*}$ is a Riesz element. Another application of Theorem \ref{theok} leads to conclude that $a^{*}$ is generalized Drazin-Riesz invertible in $\mathcal{A}$. The converse goes similarly since $(a^{*})^{*}=a$.
\end{proof}

\end{document}